\documentclass{amsart}
\usepackage{}
\usepackage{amsfonts}
\usepackage{mathrsfs}
\usepackage{amsmath}
\usepackage{paralist}
\usepackage{amssymb}
\usepackage{amsthm}
\usepackage{amscd}

\usepackage[colorlinks=true]{hyperref}
\hypersetup{urlcolor=blue, citecolor=red}

\numberwithin{equation}{section}
\newtheorem{Thm}{Theorem}[section]
\newtheorem{Lem}{Lemma}[section]
\newtheorem{Prop}{Proposition}[section]

\newtheorem{Def}{Definition}[section]
\newtheorem{Cor}{Corollary}[section]

\newtheorem*{exa}{Example}
\newtheorem{Rem}{Remark}[section]

\begin{document}
\title{Degenerate Mean Field Games with H\"{o}rmander diffusion}
\thanks{Acknowledgements: This work is supported by the NSFC under the grands 12271269 and supported by the Fundamental Research Funds for the Central Universities.}
%\thanks{All authors contributed equally.}
\author{Yiming Jiang}
\address{School of Mathematical Sciences and LPMC\\ Nankai University\\ Tianjin 300071 China}
\email{ymjiangnk@nankai.edu.cn}
\author{Jingchuang Ren}
\address{School of Mathematical Sciences \\ Nankai University\\ Tianjin 300071 China}
\email{1120200024@mail.nankai.edu.cn}
\author{Yawei Wei}
\address{School of Mathematical Sciences and LPMC\\ Nankai University\\ Tianjin 300071 China}
\email{weiyawei@nankai.edu.cn}
\author{Jie Xue}
\address{School of Mathematical Sciences \\ Nankai University\\ Tianjin 300071 China}
\email{1120200032@mail.nankai.edu.cn}
\keywords{Mean filed games; H\"{o}rmander condition; Global Schauder estimates; Weak maximum principle}
	
\subjclass[2020]{35Q89; 35K65; 35A01}
\begin{abstract}
%In this paper, we study a class of degenerate mean field game systems with H\"{o}rmander condition. The degenerate systems arise from a kind of mean field games with H\"{o}rmander diffusion, where the generic player may have a ``forbidden'' direction at some point. Here we prove the existence and uniqueness of the classical solutions in weighted H\"{o}lder spaces for the PDE systems, which describe the Nash equilibria in mean field games. There are two main difficulties in our case, which are  the lack of commutation of vector fields and the lack of fundamental solutions of degenerate operators. We overcome the first one by the localizing technique and the maximum regularity estimates to obtain the global Schauder estimates. And we construct a subsolution instead of the fundamental solution to obtain the weak maximum principle.

In this paper, we study a class of degenerate mean field game systems arising from the mean field games with H\"{o}rmander diffusion, where the generic player may have a ``forbidden'' direction at some point. Here we prove the existence and uniqueness of the classical solutions in weighted H\"{o}lder spaces for the PDE systems, which describe the Nash equilibria in the games. The degeneracy causes the lack of commutation of vector fields and the fundamental solution which are the main difficulties in the proof of the global Schauder estimate and the weak maximum principle. Based on the idea of the localizing technique and the local homogeneity of degenerate operators, we extend the maximum regularity result and obtain the global Schauder estimates. For the weak maximum principle, we construct a subsolution instead of the fundamental solution of the degenerate operators.

 % In this paper, we consider a class of degenerate partial differential systems which arise from mean field games with a forbidden direction. %We show the existence and uniqueness of the degenerate Hamiltonian-Jacobi equation and the degenerate Fokker-Planck equation in the MFG systems, respectively.
%And we give the main result that there is a unique coupling solution to the degenerate MFG systems on $[0,T]\times\mathbb{R}^2$, which implies the existence of Nash equilibrium to the corresponding MFG.
% There are two main difficulties in our case, which are  the lack of commutation of vector fields and the lack of fundamental solutions of  linear Cauchy problems. We overcome the first one by the maximum regularity estimates and Poincar\'{e} inequality. And we construct a subsolution instead of the fundamental solution to obtain the weak maximum principle.
\end{abstract}
\maketitle

\section{INTRODUCTION}	
\subsection{Statement of the problem and motivations}
In this paper, we study the degenerate mean field game (briefly, MFG) systems as follows
\begin{equation}\label{eq1.1}
\begin{cases}
 -\partial_t u-\Delta_\mathcal {X} u+H(x,\nabla_\mathcal {X}u)=F(x,m),\ &(t,x)\in (0,T)\times\mathbb{R}^2,\ \mathrm{(HJE)}\\
    \partial_t m-\Delta_\mathcal {X} m-\mathrm{div}_{\mathcal {X}}(mD_{p}H(x,\nabla_\mathcal {X} u))=0,\ &(t,x)\in (0,T)\times\mathbb{R}^2, \ \mathrm{(FPE)}\\
    u(T,x)=G(x,m_T),\  m(0,x)=m_0(x),\ &x\in\mathbb{R}^2,
\end{cases}
\end{equation}
where $\Delta_\mathcal {X}$ is a hypoelliptic operator associated to a family of H\"{o}rmander vector fields $\mathcal {X}:=\{X_1,X_2\}$, see in subsection 2.1 for definition, which is given by
\begin{equation}\label{eq1.60}
\Delta_\mathcal {X}:=\sum^{2}_{i=1}X_i^2,\ \text{with\ } X_1:=\partial_{x_1},\  X_2:=h(x_1)\partial_{x_2},
\end{equation}
and $\nabla_\mathcal {X}:=(X_1,X_2)$, $\mathrm{div}_\mathcal {X}:=X_1+X_2$
 are the corresponding gradient and divergence operators. Here $h$ is a smooth function possibly vanishing. The functions $F$ and $G$ are regular satisfying assumptions {\bf{(H3)}} and {\bf{(H4)}} below. The Hamiltonian is given by
\begin{equation}\label{eq1.22}
H(x,\nabla_\mathcal {X}u):=\frac{1}{2}\left|\nabla_\mathcal {X}u\right|^2.
\end{equation}

Note that the H\"{o}rmander vector field $\mathcal {X}$ given in \eqref{eq1.60} has Grushin structure. The generalized Grushin differential operators is given by
$$\Delta_{x}+\lambda^2(x)\Delta_{y},\ x\in\mathbb{R}^n,\ y\in\mathbb{R}^m,\ n,m\geq1,$$
%To be specific, let $\mathcal {X}=\{X_1,\cdots,X_q\}$ be a family of H\"{o}rmander vector fields, and $\mathcal {L}:=\sum^q_{i=1}X_i^2+X_0$ be the sum of squares operator.
%Such as for any $(x,y)\in\mathbb{R}^{N}$, defined $X_1=\partial_{x}$, $X_2=\lambda(x)\partial_{y}$, where $\lambda$ is some positive continuous function of $x$, and $\lambda(x)>0$ expect for at most a finite number of points.
where $\lambda(x)\in C(\mathbb{R}^n)$. More details of Grushin type operators, see in \cite{Grush}. %\cite{Grushin}.

The motivation of the MFG systems \eqref{eq1.1} is to describe Nash equilibria for the MFG in the following. The $u$ in ${\rm{(HJE)}}$ is the value  function of an optimal control problem of a generic player, where the dynamics is given by the controlled stochastic differential equation (briefly, SDE) with H\"{o}rmander diffusion
\begin{align}\label{eq1.7}
%\left\{
%\begin{array}{ll}
%dX_{s}=b(s,X_s)ds+\sigma(s,X_s)dB_s,\\
%X_t=x\in\mathbb{R}^2,
\left\{
 \begin{array}{ll}
  dX_{s}^1= \alpha_1(s,X_s)ds+\sqrt{2}dB^1_s,\qquad\qquad\qquad\ \\
  dX_{s}^2= \alpha_2(s,X_s)h(X^1_s)ds+\sqrt{2}h(X^1_s)dB^2_s,\\
   X_t^1=x_1\in\mathbb{R},\ X_t^2=x_2\in\mathbb{R}.
\end{array}
\right.
\end{align}
For any $s\in[t,T]$, set $X_s:=(X_s^1,X_s^2)$, $b(X_s,\alpha_s):=(\alpha_1(s,X_s), \alpha_2(s,X_s)h(X^1_s))$, $\sigma(X_s):=diag\big\{\sqrt{2}, \sqrt{2}h(X^1_s)\big\}$, $\alpha_s:=(\alpha_1(s,X_s),\alpha_2(s,X_s))$, and $B_s:=(B^1_s,B^2_s)$.
% When the function $h$ is vanishing, it gives a forbidden direction to the player in the games, and the degeneration in the partial differential system. A class of non-coercive first order MFG systems with forbidden direction are given by Mannucci et al. \cite{0}.

Then if the evolution of the whole population's distribution $m$ is given, each player wants to choose the optimal control $\alpha_s$ to minimize the cost function
\begin{equation}\label{4.21}
u(t,x)=\inf_{\alpha\in\mathscr{A}(t,x)}\mathbf{E}\left[\int^T_t\frac{1}{2}|\alpha_s|^2+F(X_s,m_s)ds+G(X_T,m_T)\right]
\end{equation}
where $\mathscr{A}(t,x)$ is the set of controls $\alpha$ such that $\mathbf{E}\left[\int^T_tF(X_s,m_s)ds\right]<\infty,$ and
 $$\mathbf{E}\bigg[\int^T_t|\alpha_s|^2ds\bigg]<\infty,\ \text{for any}\  (t,x)\in[0,T]\times\mathbb{R}^2.$$
In the SDE \eqref{eq1.7}, the drift coefficient $b(X_s,\alpha_s)$ and the diffusion coefficient $\sigma(X_s)$ are Lipschitz in $X_s$ uniformly in $\mathscr{A}(t,x)$. Assume that $B_s^1$ and $B_s^2$ are independent one dimension standard Brownian motions on a filtered probability space $(\Omega,\mathscr{F},\{\mathscr{F}_t\}_{t\geq0},\mathbf{P})$ satisfying the usual conditions in the stochastic analysis, see Chapter 3 in \cite{Hu}.
 Note that controls $\alpha_s$ are adapted to the filtration generated by $B_s$, valued in $\mathscr{A}(t,x)$, and the optimal feedback of each agent is given by
\begin{equation}\label{11.20}
\alpha^*(t,x)=-D_{p}H(x,\nabla_\mathcal {X} u)=-\nabla_\mathcal {X} u(t,x).
\end{equation}

Now we first concisely give a derivation of the Hamiltonian-Jacobi equation (briefly, HJE) in the MFG systems \eqref{eq1.1}. The main idea depends on dynamic programming principle and It\^{o}'s formula, which refers to Chapter 2 in \cite{LR}. For any stopping time $\tau\in[t,T]$, using It\^{o}'s formula to $u$ on $[t,\tau]$, we have
\begin{align}\label{2.2}
&\quad u(\tau, X_\tau)-u(t, X_t)\\\nonumber
%&= \int^\tau_t u_t (X_s,s)ds+\int^\tau_t  D u(X_s,s)dX_s+\frac{1}{2}\int^\tau_t \Delta u(X_s,s)d\left<X\right>_s \\
  &= \int^\tau_t \partial_s u(s, X_s)ds+\int^\tau_t\nabla u(s, X_s)dX_s+\frac{1}{2}\int^\tau_t\sum^{2}_{i=1}\partial^2_{x_i}u(s, X_s)d\big<X^i\big>_s\\\nonumber
 %  &\quad+\int^\tau_t(\partial_{x_1}u,\partial_{x_2}u)(\alpha_1 ds+\sqrt{2}dB^1_s,h(x_1)\alpha_2 ds+\sqrt{2}h(x_1)dB^2_s)\\\nonumber
 % &=&\int^{\tau}_t u_t (X_s,s)ds+\int^{\tau}_t \alpha_1 \partial_{x_1}uds+\int^{\tau}_t h(x_1)\alpha_2 \partial_{x_2}uds\\
 % &\ &+\int^{\tau}_t\sqrt{2}\partial_{x_1}udB^1_s+\int^{\tau}_t\sqrt{2}h(x_1)\partial_{x_2}udB^2_s+\int^{\tau}_t \partial^2_{x_1}u(X_s,s)d\left<B^1\right>_s\\
 %  &\ &+\int^{\tau}_t h^2 (x_1)\partial^2_{x_1}u(X_s,s)d\left<B^2\right>_s\\
 %  &= \int^\tau_t u_t (X_s,s)ds+\int^\tau_t \alpha D_G uds+\sqrt{2}\int^\tau_t  D_G udB_s+\int^\tau_t \Delta_G u(X_s,s)ds \\
 %  &=& \int^\tau_t u_t (X_s,s)+(\alpha D_G u+\Delta_G u)(X_s,s)ds+\sqrt{2}\int^\tau_t  D_G udB_s\\
   &= \int^\tau_t\partial_s u(s, X_s)+\mathcal {L}^{\alpha}_s u(s, X_s)ds+\sqrt{2}\int^\tau_t  \nabla_\mathcal {X} udB_s,
\end{align}
where $\mathcal {L}^{\alpha}_s$ is the second order differential operator which is given by
\begin{equation}\label{1.007}
  \big[\mathcal {L}^{\alpha}_s u(s,\cdot)\big](x):=\alpha(s,x) \nabla_\mathcal {X} u(s,x)+\Delta_\mathcal {X} u(s,x).
\end{equation}
%=(\alpha_1,\alpha_2)\cdot(\partial_{x_1}\varphi, h(x_1)\partial_{x_2}\varphi)+\partial^2_{x_1}\varphi+h^2 (x_1)\partial^2_{x_2}
By the dynamic programming principle, we have
\begin{equation}\label{1.08}
 u(t,x)=\inf_{\alpha\in\mathscr{A}(t,x)}\mathbf{E}\left[\int^\tau_t \frac{1}{2}|\alpha_s|^2+F(X_s,m_s)ds+u(\tau,X_\tau)\right].
\end{equation}
Because the martingale property, we have $\mathbf{E}\left[\int^\tau_t  \nabla_\mathcal {X} udB_s\right]=0$. Plugging \eqref{2.2} into \eqref{1.08}, we have
\begin{align*}
%&= \inf_{\alpha\in\mathscr{A}(t,x)}\mathbf{E}\left[\int^{\tau}_{t}H^*(u,\alpha_s)+F(X_s,m_s)ds+u(\tau, X_\tau)-u(t,x)\right] \\
   \inf_{\alpha\in\mathscr{A}(t,x)}\mathbf{E}\left[\int^{\tau}_{t}\frac{1}{2}|\alpha_s|^2+F(X_s,m_s)ds+\int^{\tau}_{t}\partial_su(s, X_s)+\mathcal {L}^{\alpha}_s u(s, X_s)ds\right] &=  0.
\end{align*}
Let $\tau=t+\delta$, divide by $\delta$ and let $\delta\rightarrow0$, we obtain the HJE
$$\partial_tu(t,x)+\Delta_\mathcal {X} u(t,x)+F(x,m)+\inf_{\alpha\in\mathscr{A}}\big\{H^*(x,\alpha)+\alpha \nabla_\mathcal {X} u(t,x)\big\}=0, $$
where $H^*(x,\alpha):=\frac{1}{2}|\alpha|^2$ is the Legendre transform of $H(x,\nabla_\mathcal {X}u)$ with respect to the second variable.
This is also can be written as
$$-\partial_tu(t,x)-\Delta_\mathcal {X} u(t,x)+H(x, \nabla_\mathcal {X} u)=F(x,m).$$

Next we give the derivation of the Fokker-Planck equation (briefly, FPE) in the MFG systems \eqref{eq1.1}, which refers to Chapter 1 in \cite{Carmona16}.
If $\varphi:\mathbb{R}^2\rightarrow\mathbb{R}$ is a $C_{\mathcal {X}}^2$ function with bounded derivatives, using It\^{o}'s formula on $[0,t]$, we get
\begin{align}\label{2.3}
  \varphi(X_t) %&= \varphi(X_0)+\int^{t}_{0}\partial_x \varphi_s dX_s+\frac{1}{2}\int^{t}_{0}\partial_{xx} \varphi_s d\left<X\right>_s\\\nonumber
%   &= \varphi(X_0)+\int^{t}_{0}\alpha_1\partial_{x_1}\varphi_s+ h(x_1)\alpha_2\partial_{x_2}\varphi_s ds +\sqrt{2}\int^{t}_{0}\partial_{x_2}\varphi_s dB^1_s \\\nonumber
%&\quad +\sqrt{2}\int^{t}_{0}h(x_1)\partial_{x_2}\varphi_s dB^2_s+\int^{t}_{0}\partial_{x_1x_1}\varphi_s
 %   +h^{2}(x_1)\partial_{x_2x_2}\varphi_s ds\\\nonumber
   &= \varphi(X_0)+\int^{t}_{0}\mathcal {L}^{\alpha}_{s}\varphi (X_s)ds+\sqrt{2}\int^{t}_{0} \nabla_\mathcal {X} \varphi(X_s) dB_s.
\end{align}
We denote the distribution of $X_t$ by $\mu_t(dx)=P(X_t\in dx)$, and use the notation
$\left<\varphi,\mu\right>:=\int_{\mathbb{R}^2}\varphi(x)\mu(dx).$
Taking expectations on both sides of \eqref{2.3}, we have
$$\left <\varphi,\mu_t\right>%= \left<\varphi,\mu_0\right>+\int^{t}_{0}\left<\mathcal {L}^{\alpha}_{s}\varphi,\mu_s\right>ds
                  %&= \left<\varphi,\mu_0\right>+\int^{t}_{0}\left<\varphi,(\mathcal {L}^{\alpha}_{s})^*\mu_s\right>ds \\
                  =\big <\varphi,\mu_0+\int^{t}_{0}(\mathcal {L}^{\alpha}_{s})^*\mu_s ds\big>,$$
%\begin{align*}
% \left <\varphi,\mu_t\right> &= \left<\varphi,\mu_0\right>+\int^{t}_{0}\left<\mathcal {L}^{\alpha}_{s}\varphi,\mu_s\right>ds \\
                  %&= \left<\varphi,\mu_0\right>+\int^{t}_{0}\left<\varphi,(\mathcal {L}^{\alpha}_{s})^*\mu_s\right>ds \\
%                  &=\big <\varphi,\mu_0+\int^{t}_{0}(\mathcal {L}^{\alpha}_{s})^*\mu_s ds\big>,
%\end{align*}
where $(\mathcal {L}^{\alpha}_{s})^*$ is the adjoint of the operator $\mathcal {L}^{\alpha}_{s}$. % so
%$$\int_{\mathbb{R}^2}\varphi\mu_t dx=\int_{\mathbb{R}^2}\varphi\bigg[\mu_0+\int^t_0 (\mathcal {L}^{\alpha}_{s})^* \mu_s ds\bigg] dx.$$
Assuming that $\mu_t$  has a density, say  $m(t,x)dx=\mu_t(dx)$, since this equality is true for all test functions $\varphi$, then the density $m$ is the solution of
\begin{equation}\label{eq2.11}
\partial_t m(t,x)=(\mathcal {L}^{\alpha}_{t})^* m(t,x),
\end{equation}
with initial condition $m_0$. It follows from \eqref{1.007} and the integration by parts that
\begin{align}\label{2.5}
 \left<\varphi,(\mathcal {L}^{\alpha}_{t})^* m\right> &=\left<\mathcal {L}^{\alpha}_{t}\varphi, m\right> \\\nonumber\nonumber
  &= \int_{\mathbb{R}^2}\big(\alpha_1 \partial_{x_1}\varphi+\alpha_2 h(x_1)\partial_{x_2}\varphi\big)m+\big(\partial^2_{x_1}\varphi+h^2(x_1)\partial^2_{x_2}\varphi\big)m dx\\\nonumber
 &= \int_{\mathbb{R}^2}\big(\partial^2_{x_1}m+h^2(x_1)\partial^2_{x_2}m\big)\varphi-\big(\partial_{x_1}(m\alpha_1) + h(x_1)\partial_{x_2}(m\alpha_2)\big)\varphi dx\\\nonumber
 &=\int_{\mathbb{R}^2}\varphi\big(-\mathrm{div}_\mathcal {X}(m\alpha)+\Delta_\mathcal {X} m\big)dx.
\end{align}
Because of the arbitrariness of the function $\varphi$, we get
\begin{equation}\label{1.008}
  (\mathcal {L}^{\alpha}_{t})^* m=\Delta_\mathcal {X} m-\mathrm{div}_\mathcal {X}(m\alpha).
\end{equation}
It follows from \eqref{eq2.11}, \eqref{1.008} and \eqref{11.20} that
$$\partial_t m-\Delta_\mathcal {X} m-\mathrm{div}_\mathcal {X}(m\nabla_{\mathcal {X}}u)=0.$$

It is worth noting that the derivation relies on the symmetry of Grushin type operators, that is $X_i^*=-X_i$, which ensures that the integration by parts in \eqref{2.5} holds.
\subsection{Research history and main results}
MFG theory is devoted to the analysis of differential games with infinitely many players. This theory has been introduced by Lasry and Lions in 2006 \cite{1} \cite{2}. At about the same time, Huang et al. \cite{3} solved the large population games independently. Then MFG has been studied extensively in many different fields. Bensoussan et al. \cite{Be} studied the MFG and mean field type control theory. Carmona and Delarue \cite{6} focused on the theory and applications of MFG by probabilistic approach. Gangbo \cite{5} developed optimal transport theory within the MFG framework.  Gomes et al. \cite{Gomes} discussed regularity theory for MFG systems either stationary or time-dependent, local or nonlocal. Cardaliaguet et al. \cite{Cardaliaguet15} obtained  the existence of classical solutions for the master equation of MFG.%, and its application to the convergence of games with a large population to a MFG.
%The classical MFG systems have been studied widely, and there are wide ranges of applications for MFG in economic and biology.

Degenerate MFG systems have much fewer references than the classical one. The hypoelliptic operator, which satisfies the H\"{o}rmander condition, is an intermediate case between uniformly elliptic and general degenerate elliptic. The hypoelliptic MFG was addressed by Dragoni and Feleqi \cite{16}. They studied the ergodic systems with H\"{o}rmander diffusions  and obtained a unique classical solution in the weighted spaces. Feleqi et al. \cite{TT} obtained regularities for hypoelliptic MFG based on the theory of eigenvalue problems. Furthermore, Camilli \cite{F} studied the quadratic MFG with Langevin diffusion and obtained weak solutions to the coupled kinetic systems with uncoupled terminal data. % Then using the kinetic regularity of hypoelliptic equations, %
%see also Mimikos-Stamatopoulos \cite{N}.
For the general degenerate MFG, Cardaliaguet et al. \cite{C} obtained the existence and uniqueness of suitably defined weak solutions for parabolic MFG systems with degenerate diffusion and local coupling%, and they showed the stability of solutions with respect to the data, so that the degenerate case can be approximated by a uniformly parabolic perturbation
. Then Ferreira et al. \cite{FT} extended the results to a wide class of time-dependent degenerate MFG. Moreover, Cardaliaguet et al. \cite{Com} developed a new notion of weak solution for the MFG with common noise and degenerate idiosyncratic noise. %Hence these results can not be directly applied to the .
For the MFG systems \eqref{eq1.1}, the degenerate operator is hypoelliptic with Grushin structure.

Hypoelliptic operators enjoy more regularizing properties than general degenerate elliptic operators. H\"{o}rmander \cite{Hormander} introduced the H\"{o}rmander condition to ensure the regularity of the degenerate sum of square operators. For the  quasi-linear  subelliptic problems of H\"{o}rmander type,  Xu and Zuily \cite{Xu-56} have obtained the existence and higher interior regularity by the Rothschild-Stein lifting theorem. Further, the corresponding regularity results of the parabolic degenerate equation are obtained by Yan \cite{Yan2}. For  the heat kernel representation of the solution, Bonfiglioli et al. \cite{heat} considered non-divergence form operators on stratified groups, and Bramanti et al. \cite{Bramanti10} considered operators which are degenerate inside a compact set.
In this paper, we obtain the H\"{o}lder regularity for the (HJE) in \eqref{eq1.1}.

The references mentioned above motivate us to discover the existence and uniqueness of the coupling classical solution in the weighted H\"{o}lder spaces for the degenerate MFG systems \eqref{eq1.1} with the Grushin structure.

Now we list the notions and assumptions as follows.

Let $\mathcal {P}_{1}$ be the set of Borel probability measures $m$ on $\mathbb{R}^2$, such that $\int_{\mathbb{R}^2}|x|dm<\infty$, and endowed with the Kantorovitch-Rubinstein distance
\begin{equation}\label{distance}
 d_1(\mu,\nu):=\inf_{\gamma\in\Pi(\mu,\nu)}\int_{\mathbb{R}^2}|x-y|d\gamma(x,y),
\end{equation}
where $\Pi(\mu,\nu)$ is the set of Borel probability measures on $\mathbb{R}^2$ such that $\gamma(E\times\mathbb{R}^2)=\mu(E)$,
 and $\gamma(\mathbb{R}^2\times E)=\nu(E)$ for any Borel set $E\subset\mathbb{R}^2$. For more details of this distance can refer to  \cite{Cardaliaguet}.

 Let $\mathcal {C}$ be the set of maps $\mu\in C([0,T];\mathcal {P}_{1})$, such that $\sup_{t\in[0,T]}\int_{\mathbb{R}^{2}}|x|^2dm\leq C$,
\begin{equation}\label{eq2.7}
\sup_{s\neq t}\frac{d_1(\mu_s,\mu_t)}{|t-s|^{\frac{1}{2}}}\leq C.
\end{equation}
Then $\mathcal {C}$ is a convex closed subset of $C([0,T];\mathcal {P}_{1})$. In fact, it is compact in $\mathcal {P}_{1}$, more properties of the space of probability measures refer to Chapter 5 in \cite{Cardaliaguet}.

Throughout this paper, $C$ is a generic positive constant which may differ from line by line, and the following hypotheses are required.

\noindent{\bf(H1)} The smooth function $h(x_1)$ satisfies $\sup_{x_1\in\mathbb{R}}|h(x_1)|<\infty$.

\noindent{\bf(H2)} Set $\mathcal {N}(h):=\{y\in\mathbb{R}|h(y)=0\}$. Then $\mathcal {N}(h)\neq\emptyset$, and for any $y\in\mathcal {N}(h)$, $y$
 is isolated. Moreover, $\exists\  \kappa\in\mathbb{N}$, s.t. $y\not\in\mathcal {N}(h^{(\kappa)}) $, i.e., $h^{(\kappa)}(y)\neq0$, where $h^{(\kappa)}$ is the $\kappa$-order derivative of $h$.

\noindent{\bf(H3)} The functions $F(x,m)$ and $G(x,m_T)$ are Lipschitz continuous in $\mathbb{R}^2\times\mathcal {P}_{1}$, that is for any $(x,\mu),(y,\nu)\in\mathbb{R}^2\times\mathcal {P}_{1}$,

{\rm{(A1)}} $|F(x,\mu)-F(y,\nu)|\leq C \big(d_{C}(x,y)+d_1(\mu,\nu)\big)$,

{\rm{(A2)}} $|G(x,\mu_T)-G(y,\nu_T)|\leq C\big(d_{C}(x,y)+d_1(\mu_T,\nu_T)\big)$,

where $d_C$ is given in Definition \ref{2.6} below.

\noindent{\bf(H4)} The functions $F(x,m)$ and $G(x,m_T)$ are uniform regular in weighted H\"{o}lder space, i.e. for some $0<\alpha<1$, $F(x,m)$ and $G(x,m_T)$ are bounded in $C^{0,\alpha}_{\mathcal {X}}(\mathbb{R}^2)$, and $C^{2,\alpha}_{\mathcal {X}}(\mathbb{R}^2)$ uniformly with respect to the measure $m\in\mathcal {P}_{1}$.

\noindent{\bf(H5)} The functions $F(x,m)$ and $G(x,m_T)$ are monotonically increasing with respect to the measure $m$, that is for any $\mu,\nu\in\mathcal {P}_{1}$,

{\rm{(A3)}} $\int_{\mathbb{R}^2}F(x,\mu)-F(x,\nu)d(\mu-\nu)\geq 0$,

{\rm{(A4)}} $\int_{\mathbb{R}^2}G(x,\mu_T)-G(x,\nu_T)d(\mu_T-\nu_T)\geq 0$.

\noindent{\bf(H6)} The probability measure $m_0$ is absolutely continuous with respect to the Lebesgue measure satisfying $\int_{\mathbb{R}^2}|x|^2 dm_0<\infty$, which has a $ C^{2,\alpha}$ continuous density, still denoted $m_0$.

%\noindent{\bf{Example}}
\begin{exa}
Easy examples of $h(x_1)$ satisfying assumptions {\bf(H1)} and {\bf(H2)}, are $$h(x_1)=\frac{x_1}{\sqrt{x^2_1+1}},$$
 which has been studied as a degenerate diffusion term in \cite{27}, and
$$h(x_1)=\sin x_1.$$
More general example is given by
$$h(x_1)=\frac{x_1^a}{(x_1^b+C)^{\frac{a}{b}}},$$
where $C>0$, and $a,\ b$ are positive integers.
The function $h$ in these examples indicates the degenerate operators in MFG systems \eqref{eq1.1}, holding the H\"{o}rmander condition, and bounded in $\mathbb{R}$.
\end{exa}
Let $\mathcal {X}=\{X_1,\cdots,X_q\}$, $1\leq q\leq n$, be a family of H\"{o}rmander vector fields on $\mathbb{R}^n$ with Grushin structure, we define the parabolic non-divergence operators
\begin{equation}\label{001}
  \mathcal {H}:=\partial_t-\sum^{q}_{i,j=1}a_{ij}(t,x)X_iX_j+\sum^{q}_{i=1}b_{i}(t,x)X_i+c(t,x),
\end{equation}
for any $(t,x)\in(0,T)\times\mathbb{R}^n$. Let $a_{ij},b_i,c\in C^{0,\alpha}_{\mathcal {X}}([0,T]\times\mathbb{R}^n)$, for some $\alpha\in(0,1)$, and $A :=\big\{a_{ij}(t,x)\big\}^{q}_{i,j=1}$ be a symmetric, uniformly positive definite matrix on $[0,T]\times\mathbb{R}^n$, that is
\begin{equation}\label{1.14}
  C_{\xi}^{-1}|\xi|^2\leq\sum^{q}_{i,j=1}a_{ij}\xi_i\xi_j\leq C_{\xi}|\xi|^2,\ \text{for\ some\ }C_{\xi}>0.
\end{equation}

Here below we state the main results of this paper.

 Consider the quasi-linear HJE in MFG systems \eqref{eq1.1} with the fixed density $\bar{m}\in C([0,T];\mathcal {P}_{1})$. Set
  \begin{equation}\label{U}
   U_T:=(0,T)\times\mathbb{R}^2.
\end{equation}
We use the Hopf transform, setting $w=e^{-\frac{u}{2}}$. It is clearly that $u$ is a solution of the following quasi-linear equation
 \begin{equation}\label{eq1.3}
\left\{
 \begin{array}{ll}
  \partial_t u-\Delta_{\mathcal {X}} u+\frac{1}{2}| \nabla_{\mathcal {X}} u|^2=F(x,\bar{m}),\quad\quad\ \ (t,x)\in U_{T}, \\
     u(0,x)=G(x,\bar{m}_T),\qquad\qquad\qquad\qquad\ \ \ x\in\mathbb{R}^2,
 \end{array}
\right.
\end{equation}
 if and only if $w$ is a solution of the linear equation
\begin{equation}\label{2.0}
\left\{
  \begin{array}{ll}
    \partial_t w-\Delta_{\mathcal {X}}w+\frac{1}{2}F w=0,\quad\quad\ \ (t,x)\in U_T,  \\
    w(0,x)=e^{-\frac{G(x,\bar{m}_T)}{2}},\qquad\quad\quad\ \ x\in\mathbb{R}^2.
  \end{array}
\right.
\end{equation}

First, we prove the weak maximum principle (briefly, WMP) on $U_T$ for the solution of the linear Cauchy problem with weak regularity.
\begin{Lem}\label{WMP4}{\bf{(WMP for linear Cauchy problems)}}
Given vector fields $\mathcal {X}$ defined as \eqref{eq1.60} satisfy assumptions {\bf{(H1)}} and {\bf{(H2)}}. If $c(t,x)$ is bounded, and $u\in C_{\mathcal {X}}^{2}(U_{T})\cap C(\bar{U}_{T})$ satisfies
\begin{equation}\label{12}
\left\{
  \begin{array}{ll}
    \partial_t u(t,x)-\Delta_{\mathcal {X}}u(t,x)+c(t,x)u(t,x)\leq0,\quad\ (t,x)\in U_{T},\\
     u(0,x)=g(x),\qquad\qquad\qquad\qquad\qquad\qquad\ \ x\in\mathbb{R}^2,
  \end{array}
\right.
\end{equation}
with the growth estimate
\begin{equation}\label{2.11}
  u(t,x)\leq Ce^{a(d_{C}(x))^2},
\end{equation}
for any $(t,x)\in\bar{U}_{T}$, and constants $C ,a>0$, then
$$\sup_{\bar{U}_{T}}u\leq\sup_{\mathbb{R}^2}g^{+}.$$
\end{Lem}

Next,  by the localizing technique, we obtain the global Schauder estimates for the general linear operator $\mathcal {H}$ defined in \eqref{001}.
%based on the known internal results given in Lemma \ref{Sch}, we obtain the global Schauder estimates, instead of partial Schauder estimates in \cite{09}.
\begin{Lem}\label{1.1}
{\bf{(Schauder estimate)}} Given vector fields $\mathcal {X}$ defined as \eqref{eq1.60} satisfy assumptions {\bf{(H1)}} and {\bf{(H2)}}. If $u \in C^{2,\alpha}_{\mathcal {X}}(\bar{U}_{T})$ is a solution of the linear problem
\begin{equation}\label{1.8}
\left\{
  \begin{array}{ll}
    \mathcal {H} u=f, \qquad\qquad\qquad\quad (t,x)\in U_{T}, \\
     u(0,x)=u_0(x),\qquad\quad\ \ \ x\in\mathbb{R}^2, \\
  \end{array}
\right.
\end{equation}
where $\mathcal {H}$ defined in \eqref{001}, $f\in C^{0,\alpha}_{\mathcal {X}}(\bar{U}_{T})$, and $u_0\in C^{2,\alpha}_{\mathcal {X}}(\mathbb{R}^2)$ for some $\alpha\in(0,1)$, then there exists a constant $C>0$ depending on $U,\mathcal {X},\alpha, C_\xi$ such that
\begin{equation}\label{3.35}
\|u\|_{C^{2,\alpha}_{\mathcal {X}}(U_T)}\leq C\left(\|f\|_{C^{0,\alpha}_{\mathcal {X}}(U_T)}+\|\mathcal {H}u_0\|_{C^{0,\alpha}_{\mathcal {X}}(U_T)}+\|u\|_{L^{\infty}(U_T)}\right).
\end{equation}
%Furthermore, for any $(t,x)\in U_T$, $R>0$, denote $U_R:=(t-R^2,t+R^2)\times B_R(x)\subset U_T$, then we have
%\begin{equation}\label{1.9}
%  \|u\|_{C^{2,\alpha}_{\mathcal {X}}(U_R)}\leq C\left(\|f\|_{C_{\mathcal {X}}^{0,\alpha}(U_R)}+\|u_0\|_{C_{\mathcal {X}}^{2,\alpha}(U_{R})}+\|u\|_{L^{L^{\infty}(U_T)}}(U_R)}\right).
%\end{equation}
\end{Lem}

Then, we obtain the existence and uniqueness of HJE in the MFG systems \eqref{eq1.1} by the vanishing viscosity method. Combined with Lemma \ref{WMP4} and Lemma \ref{1.1}, we obtain the regularity of HJE in weighted H\"{o}lder space.

\begin{Thm}\label{1.2}{\bf{(Existence and uniqueness of HJE)}}
%Given vector fields $\mathcal {X}$ defined as \eqref{eq1.60}, satisfy assumptions {\bf{(H1)}} and {\bf{(H2)}}.
Under assumptions {\bf{(H1)}}-{\bf{(H4)}}, fixed the probability measure $\bar{m}\in C([0,T];\mathcal {P}_{1})$, there exists a unique solution $u$ of the HJE \eqref{eq1.3}.
 Moreover, $u\in C^{2,\alpha}_{\mathcal {X}}(\bar{U}_{T})$, and satisfies
\begin{equation}\label{21.2}
  \|u\|_{C^{2,\alpha}_{\mathcal {X}}(\bar{U}_{T})}\leq C.
\end{equation}
\end{Thm}

Consider the linear FPE in MFG systems \eqref{eq1.1}  with the optimal feedback control given in Theorem \ref{1.2}. By Cantor's diagonal argument, we obtain the existence and uniqueness of  the FPE.
\begin{Thm}\label{1.3}{\bf{(Existence and uniqueness of FPE)}}
%Given vector fields $\mathcal {X}$ defined as \eqref{eq1.60} satisfy assumptions {\bf{(H1)}} and {\bf{(H2)}}.
Under  assumptions {\bf{(H1)}}-{\bf{(H4)}}, {\bf{(H6)}}, fixed $u\in C^{2,\alpha}_{\mathcal {X}}(\bar{U}_{T})$, there exists a unique solution $m$ of the FPE
 \begin{equation}\label{eq1.4}
\left\{
 \begin{array}{ll}
  \partial_t m-\Delta_{\mathcal {X}} m-\nabla_{\mathcal {X}}u\nabla_{\mathcal {X}}m- m\Delta_{\mathcal {X}}u=0,\quad\quad (t,x)\in U_{T}, \\
  m(0,x)=m_0(x),\qquad\qquad\qquad\qquad\qquad\qquad x\in\mathbb{R}^2.
 \end{array}
\right.
\end{equation}
Moreover, $m\in C^{2,\frac{\alpha}{\kappa+1}}_{\mathcal {X}}(\bar{U}_{T})$.
\end{Thm}
Finally, combined with Theorem \ref{1.2} and Theorem \ref{1.3}, we obtain the existence of the coupling solution for the degenerate MFG \eqref{eq1.1} in the weighted H\"{o}lder space by Schauder fixed point theorem.
\begin{Thm}\label{1.4}{\bf{(Existence and uniqueness of MFG)}}
 Under assumptions {\bf{(H1)}}-{\bf{(H6)}}, given the vector fields $\mathcal {X}$ defined as \eqref{eq1.60}, there exists a unique coupling solution $(u,m)\in C^{2,\alpha}_{\mathcal {X}}([0,T]\times\mathbb{R}^2)\times C([0,T];\mathcal {P}_{1})$ to degenerate MFG systems \eqref{eq1.1}.
\end{Thm}

The main contributions of this paper are summarized in the following three points. First, we study a class of degenerate MFG systems \eqref{eq1.1} with H\"{o}rmander condition in $(0,T)\times \mathbb{R}^2$. The degenerate systems arise from a kind of MFG with H\"{o}rmander diffusion, where the generic player may have a ``forbidden'' direction at some point. Here we prove the existence and uniqueness of the classical solutions in weighted H\"{o}lder spaces for the PDE systems, which describe the Nash equilibria in the MFG.
%In contrast to ergodic MFG systems in \cite{16}, we obtain the regularity results for the MFG system \eqref{eq1.1}.
Second, as prior estimates, the global Schauder estimates play a critical role in the proof of the existence and regularity of the HJE \eqref{eq1.3}. The difficulty here is the lack of commutation of vector fields, which leads to the derivative of the solution not satisfying the original equation. Based on the idea of the localizing technique and %the method of freezing the coefficients,
the local homogeneity of degenerate operators, we extend the maximum regularity result and obtain the global Schauder estimates. Third, the weak maximum principle(WMP) ensures a unique positive solution of the linear degenerate parabolic Cauchy problem derived from the HJE by the Hopf transform. The difficulty  here is the lack of fundamental solutions, such that we can not directly apply classical methods on unbounded regions to the degenerate case. To obtain the WMP, we construct a subsolution instead of the fundamental solution of the degenerate operators.

The rest of the paper is organized as follows. In Section 2, we set up the weighted spaces induced by H\"{o}rmander vector fields, and give some known results. In Section 3, we prove the WMP given in Lemma \ref{WMP4} and the Schauder estimates given in Lemma \ref{1.1}. Then we prove the existence and uniqueness of the HJE given in Theorem \ref{1.2}. In Section 4, we prove the existence and uniqueness for the FPE given in Theorem \ref{1.3}. In Section 5, by Schauder fixed point theorem, we obtain the existence and uniqueness of the degenerate MFG systems \eqref{eq1.1} given in Theorem \ref{1.4}.
\section{PRELIMINARIES AND KNOWN RESULTS}
\subsection{Set up the weighted spaces with the H\"{o}rmander condition}
Let $\mathcal {X}=\{X_1,X_2,\ldots,X_q\}$, $q\leq n$, a family of real smooth vector fields which are defined in some open set $\Omega\subset\mathbb{R}^n$, and the heat-type operator in $\mathbb{R}^{n+1}$ defined in \eqref{001}.
Set $X_0=\partial_t$. Let us assign to each $X_i$ a weight $p_i$, saying that
$$p_0=2,\ \text{and}\ p_i=1\ \text{for}\ i=1,2,\cdots,q.$$
 For any multi-index $I=(i_1,i_2\cdots,i_s)$, $1\leq i_j\leq q$, we define the weight of $I$ is $$|I|=\sum_{j=1}^{s}p_{i_j}.$$
 For any couple of vector fields $X,Y$, define their commutator as
$$[X,Y]:=XY-YX.$$
Now set $X^I :=X_{i_1}X_{i_2}\cdots X_{i_s}$, and
$$X^{[I]} :=[X_{i_s},[X_{i_{s-1}},\ldots[X_{i_3},[X_{i_2},X_{i_1}]]\ldots]].$$
We will say that $X^{[I]}$ is a commutator with weight $|I|$.
Furthermore, we say that H\"{o}rmander vector fields $\mathcal {X}$ with the H\"{o}rmander index $\tilde{\kappa}$ if  these vector fields, together with their commutators of weight $|I|\leq\tilde{\kappa}\in\mathbb{N}$, span the tangent space at every point of $\Omega$. For more details on H\"{o}rmander  vector fields see in \cite{Hormander}.

Recalling some basic definitions induced by a family of vector fields $\mathcal {X}$.
\begin{Def}\label{2.6}{\rm{(Carnot-Carath\'{e}odory distance)}}
For any $x,\ y\in\Omega$, let
\begin{equation}\label{d1}
d_{C}(x,y)=\inf\{l(\gamma)|\ \gamma:[0,l(\gamma)]\rightarrow \Omega\ \text{is}\ \mathcal{X}\text{-subunit},\ \gamma(0)=x,\ \gamma(l(\gamma))=y\},
\end{equation}
where we call $\mathcal{X}$-subunit any absolutely continuous curve $\gamma$ such that
$$\gamma'(t)=\sum^{q}_{j=1}\lambda_j(t)X_j(\gamma(t)),\ \text{a.e.,\ with}\ \sum^{q}_{j=1}\lambda^2_j(t)\leq1\ a.e..$$
Then $d_{C}$ is a distance called the Carnot-Carath\'{e}odory distance, or CC-distance.
\end{Def}

\begin{Rem}
According to Chow's theorem \cite{Nagel85}, if the vector fields satisfy the H\"{o}rmander condition, the set defined in \eqref{d1} is nonempty, so that $d_{C}$ is finite and continuous w.r.t. the original Euclidean topology included on $\mathbb{R}^n$, see in \cite{Montgomery}.
\end{Rem}

For any $x,y\in \Omega$, a known result in \cite{Fefferman} states that
\begin{equation}\label{2.7}
C^{-1}|x-y|\leq d_{C}(x,y)\leq C|x-y|^{\frac{1}{\tilde{\kappa}}},
\end{equation}
where $\tilde{\kappa}$ is the H\"{o}rmander index. Let $\Omega'\subset\Omega$ be a fixed domain, for any $x\in\Omega'$, set
 $$B_R(x)=\{y\in \Omega:d_{C}(x,y)<R\}.$$
There exists some $Q>0$, called the homogenous dimension, such that for all $R>0$,
\begin{equation}\label{2.8}
  C^{-1}R^{Q}\leq|B_{R}(x)|\leq CR^Q.
\end{equation}
For more details on $CC$-distance see in \cite{Nagel85}. %, \cite{Calle} %and \cite{condition}.

Now let us introduce the parabolic Carnot-Carath\'{e}odory distance $d_P$ corresponding to $d_{C}$ in $\mathbb{R}^{n+1}$, namely,
\begin{equation}\label{9.29}
d_P((t,x),(s,y))=\sqrt{d_{C}(x,y)^2+|t-s|},
\end{equation}
and the corresponding ball is given by
\begin{equation}\label{2.4}
B_P((t,x),R)=\big\{(s,y)\in\mathbb{R}^{n+1}:d_P((t,x),(s,y))<R\big\}.
\end{equation}
It's also not difficult to prove $d_P((t,x),(s,y))$ defined as above is a distance on $\mathbb{R}\times \Omega$, see Lemma 3.4 in \cite{Bramanti}.

%Then we introduce parabolic  weighted H\"{o}lder and Sobolev spaces associated to the H\"{o}rmander vector fields $\mathcal {X}$.
 Let $\Omega_T:=(0,T)\times \Omega\subseteq \mathbb{R}^{n+1}$ be an open set, for any $\alpha\in(0,1)$, the seminorm of  H\"{o}lder space is defined as
\begin{equation}\label{8.05}
[u]_{C^{0,\alpha}_{\mathcal {X}}(\Omega_T)}:=\sup\left\{\frac{|u(t,x)-u(s,y)|}{d_{P}((t,x),(s,y))^{\alpha}}:(t,x),(s,y)\in \Omega_T,(t,x)\neq(s,y)\right\},
\end{equation}
the norm of weighted H\"{o}lder space is defined as
$$\|u\|_{C^{0,\alpha}_{\mathcal {X}}(\Omega_T)}:=[u]_{C^{0,\alpha}_{\mathcal {X}}(\Omega_T)}+\|u\|_{L^{\infty}(\Omega_T)}.$$
\begin{Def}\label{2.06}{\rm{(Weighted H\"{o}lder space)}}
Let $r$ be a nonnegative integer, the set
$$C^{r,\alpha}_{\mathcal {X}}(\Omega_T):=\{u:\Omega_T\rightarrow\mathbb{R},\ \text{such\ that}\ \|u\|_{C^{r,\alpha}_{\mathcal {X}}(\Omega_T)}<\infty\},$$
is called a weighted H\"{o}lder space related to H\"{o}rmander vector fields $\mathcal {X}$ with the norm
\begin{equation}\label{21.9}
  \|u\|_{C^{r,\alpha}_{\mathcal {X}}(\Omega_T)}:=\sum_{|I|+2j\leq r}\|\partial_t^{j}X^{I}u\|_{{C^{0,\alpha}_{\mathcal {X}}(\Omega_T)}},
\end{equation}
where $I$ is multi-index, $j$ is a nonnegative integer.
\end{Def}%It's worth noting that the set of map $u:\Omega_T\rightarrow\mathbb{R}$ such that the derivatives $\partial^j_t D^s_x u$ exist for any pair $(j,s)$ with $2j+s\leq r$, and such that these derivatives are bounded and $\alpha$-H\"{o}lder continuous in space, and $\alpha/2$-H\"{o}lder continuous in time.
\begin{Def}\label{2.07}{\rm{(Weighted Sobolev space)}}
For $1\leq p<\infty$, the set
$$W^{r,p}_{\mathcal {X}}(\Omega_T):=\big\{u\in L^p (\Omega_T):\partial^{j}_{t}X^I u\in L^p(\Omega_T),\ \text{for\ any\ } I\ \text{satisfies}\ |I|+2j\leq r\big\},$$
is called a weighted Sobolev space related to H\"{o}rmander vector fields $\mathcal {X}$ with the norm
$$\|u\|_{W^{r,p}_{\mathcal {X}}(\Omega_T)}:=\bigg(\sum_{|I|+2j\leq r}\iint_{\Omega_T}|\partial^{j}_{t}X^I u|^p dxdt\bigg)^{\frac{1}{p}}.$$
For $p=2$, we write $H^{r}_{\mathcal {X}}(\Omega_T)$ instead of $W^{r,p}_{\mathcal {X}}(\Omega_T)$.
\end{Def}
\begin{Def}\label{2.08}{\rm{(Campanato space)}} Let $p\geq1$ and $\lambda\geq0$. A function $u\in L^p_{loc}(\Omega_T)$ is said to belong to the Campanato space $\mathcal {L}^{p,\lambda}_{\mathcal {X}}(\Omega_T)$ if
$$\|u\|_{\mathcal {L}^{p,\lambda}_{\mathcal {X}}(\Omega_T)}=[u]_{p,\lambda,\Omega_T}+\|u\|_{L^p(\Omega_T)}<\infty,$$
where $$[u]_{p,\lambda,\Omega_T}=\sup_{(t,x)\in\Omega_T,R\in(0,d_0)}\big(R^{-\lambda}\iint_{\Omega_T\cap B_P((t,x),R)}|u-\bar{u}_P|^{p}dxdt\big)^{\frac{1}{p}},$$
and $\bar{u}_{P}$ is the average of $u$ on $B_P((t,x),R)$, $d_0=\min\{diam(\Omega_T),R_0\}$, for fixed $R_0$.
\end{Def}
\begin{Rem}
For any $1\leq p<\infty$, the following embeddings hold
\begin{equation}\label{10.19}
  C^{\tilde{\kappa} r,\alpha}_{\mathcal {X}}(\Omega)\hookrightarrow C^{r,\frac{\alpha}{\tilde{\kappa}}}(\Omega),\quad W^{\tilde{\kappa} r,p}_{\mathcal {X}}(\Omega)\hookrightarrow W^{r,p}(\Omega),
\end{equation}
where $\tilde{\kappa}$ is the H\"{o}rmander index.
The first is proved in \cite{Xu90} and the second is proved in \cite{Xu92}.
More properties about continuity, compactness and embedding in the weighted spaces, see in \cite{Bramanti10}.
\end{Rem}
\begin{Rem}
We claim that the assumption {\bf{(H2)}} ensures that the vector fields $\mathcal {X}$ given in \eqref{eq1.60} satisfy H\"{o}rmander condition with H\"{o}rmander index
\begin{equation}\label{2.78}
\tilde{\kappa}=\kappa+1.
\end{equation}
In fact, since $[X_1,X_2]%=X_1X_2-X_2X_1=\partial_{x_1}(h(x_1)\partial_{x_2})-h(x_1)\partial_{x_2}(\partial_{x_1})
=h'(x_1)\partial_{x_2}$ and $\left[X_2,\left[X_1,X_2\right]\right]=0,$ it is sufficient to prove
\begin{equation}\label{1.222}
  \left[X_{i_{\kappa}},\left[X_{i_{\kappa-1}},\cdots\left[X_{i_{1}},X_{2}\right]\cdots\right]\right]=h^{(\kappa)}(x_1)\partial_{x_2},\ \text{with}\ i_{1}=\cdots=i_{\kappa}=1.
\end{equation}
Since $h{^{(\kappa)}}(x_1)\neq0$ by {\bf{(H2)}}, we have $|I|=\kappa+1$, and the claim is valid.

Here we prove the \eqref{1.222} by induction. The equality \eqref{1.222} clearly holds with $\kappa=1$. Assume that \eqref{1.222} is valid for $\kappa-1$, by calculations, we have
$$\left[X_{1},\left[X_{i_{\kappa-1}},\cdots\left[X_{i_{1}},X_2\right]\cdots\right]\right]
=\big[X_{1},h^{(\kappa-1)}(x_1)\partial_{x_2}\big]=h^{(\kappa)}(x_1)\partial_{x_2}.$$
\end{Rem}
\begin{Rem}
If the vector fields $\mathcal {X}$ defined as \eqref{eq1.60} satisfy assumptions {\bf{(H1)}} and {\bf{(H2)}}, then the corresponding homogenous dimension is given by
\begin{equation}\label{2.24}
  Q=\kappa+2,
\end{equation}
which refers to  M\'{e}tivier condition in \cite{8}.
\end{Rem}

\subsection{Known results of weighted Sobolev space}
First, we show the known result on internal Schauder estimates.
\begin{Lem}\label{Sch}{\rm{(Theorem 4.1 in \cite{Bramanti10})}}
Let $\mathcal {X}=\{X_1,X_2,\ldots,X_q\}$ be a system of smooth real vector fields and satisfying H\"{o}rmander condition on $U_{T}$ defined in \eqref{2.0}. Assume the operator $\mathcal {H}$ defined in \eqref{001} with $a_{ij},b_i,c\in C_{\mathcal {X}}^{k,\alpha}(U_{T})$ for some integer $k\geq0$ and some $\alpha\in(0,1)$. Then for any domain $U'_T\Subset U_T$, there exists a constant $C>0$ depending on $U_T,U'_T,\mathcal {X},\alpha, C_\xi$ defined in \eqref{1.14} such that for each $u \in C^{2,\alpha}_{\mathcal {X}}(U'_T)$ with $\mathcal {H}u\in C_{\mathcal {X}}^{0,\alpha}(U_{T})$ one has
\begin{equation}\label{Sch1}
  \|u\|_{C_{\mathcal {X}}^{2,\alpha}(U'_T)}\leq C\left\{\|\mathcal {H}u\|_{C_{\mathcal {X}}^{0,\alpha}(U_{T})}+\|u\|_{L^{\infty}(U_{T})}\right\}.
\end{equation}
\end{Lem}
Second, we give a parabolic version of the equivalence property.
%\begin{Lem}{\rm{(Proposition 2.1 in \cite{Xu-56})}}
%Let $u$ be in $L^2(\Omega)$, $\bar{U}_{T}_{R}$ denotes the average of $u$ on $B_R$, then the following conditions are equivalent.

%\noindent a)$u\in C^{0,\alpha}_{loc}(\Omega)$,

%\noindent b)One can find positive constant $R_0$ and $C$ such that for any $R$ in $(0,R_0]$, and $x_0\in \Omega$, such that one has
%$$\int_{B_R}|u(x)-\bar{U}_{T}_R|^2dxdt\leq C|B_R|R^{2\alpha}.$$
%\end{Lem}
%Since the metric $d_P$ is continuous, we can obtain the corresponding in parabolic case, the proof is very close to that in Euclidian case.
\begin{Lem}\label{A.3.}{\rm{(Lemma 2.3 in \cite{Yan2})}}
The spaces $\mathcal {L}^{2,Q+2+2\alpha}_{\mathcal {X}}(\Omega_T)$ and $C^{0,\alpha}_{\mathcal {X}}(\Omega_T)$, $0<\alpha<1$, are topologically and algebraically isomorphic.
%Let $u$ be in $L^2((0,T)\times  \Omega')$ for any bounded $\Omega'$, then the following conditions are equivalent.
%\begin{itemize}
%  \item [\rm{(i)}] $u\in C^{0,\alpha}_{loc}((0,T)\times\Omega')$,
 % \item [\rm{(ii)}] One can find positive constant $R_0$ and $C$ such that for any $R$ in $(0,R_0)$, and $(t_0,x_0)\in (0,T)\times\Omega'$, such that for $B_P$ defined in \eqref{2.4}, it follows
%$$\iint_{B_P((t_0,x_0),R)}|u(t,x)-\bar{u}_P|^2dxdt\leq C|B_P((t_0,x_0),R)|R^{2\alpha},$$
%where $\bar{u}_{P}$ denotes the average of $u$ on $B_P((t_0,x_0),R)$.
%\end{itemize}
\end{Lem}
Third, we give the interpolation inequality in weighted H\"{o}lder space.
\begin{Lem}\label{IE}{\rm{(Proposition 2.2 in \cite{Xu92})}}
Suppose $j+\beta<k+\alpha$, $j,k\in\mathbb{N}$, $\alpha,\beta\in[0,1]$, and $u\in C^{k,\alpha}_\mathcal {X}(\Omega)$. Then for any $\varepsilon>0$, there exists a constant $C(\varepsilon,j,k,\Omega)$ such that
\begin{equation}\label{9.27}
 \|u\|_{C^{j,\beta}_\mathcal {X}(\Omega)}\leq \varepsilon\|u\|_{C^{k,\alpha}_\mathcal {X}(\Omega)}+C\|u\|_{L^{\infty}(\Omega)}.
\end{equation}
\end{Lem}
 Then  we give a parabolic version of the interpolation inequality as follows.
\begin{Cor} \label{P}
Suppose $j+\beta<k+\alpha$, $j,k\in\mathbb{N}$, $\alpha,\beta\in[0,1]$, and $u\in C^{k,\alpha}_\mathcal {X}(\Omega_T)$. Then for any $\varepsilon>0$, there exists a constant $C(\varepsilon,j,k,\Omega_T)$ such that
\begin{equation}\label{9.28}
 \|u\|_{C^{j,\beta}_\mathcal {X}(\Omega_T)}\leq \varepsilon\|u\|_{C^{k,\alpha}_\mathcal {X}(\Omega_T)}+C\|u\|_{L^{\infty}(\Omega_T)}.
\end{equation}
\end{Cor}
\begin{proof}
By the definition of $d_P$ given in \eqref{9.29}, for any $(t,x),(s,y)\in\Omega_T$, we have
 $$\frac{|u(t,x)-u(s,y)|}{d_P((t,x),(s,y))^\alpha}\leq
 \frac{|u(t,x)-u(t,y)|}{d_C(x,y)^\alpha}+\frac{|u(t,y)-u(s,y)|}{|t-s|^\frac{\alpha}{2}}.$$

Since the case of the $t$-coordinate can be treated in a similar manner as the classical proof, combined with \eqref{9.27}, the result follows.

\end{proof}
Next, we quote from \cite{G} the extended version of iterative lemma.
\begin{Lem}{\rm{(Lemma 2.1 in \cite{G})\label{A.4.}}}
Let $\phi(t)$ be a nonnegative and nondecreasing function. Suppose that
$$\phi(\rho)\leq C_1\left[\left(\frac{\rho}{R}\right)^\alpha+\varepsilon\right]\phi(R)+C_2R^\beta,$$
for all $\rho< R\leq R_0$, with $C_1,\ \alpha,\ \beta$ nonnegative constants, $\beta<\alpha$. Then there exists constant $\varepsilon_0=\varepsilon_0(C_1,\ \alpha,\ \beta)$ such that if $\varepsilon<\varepsilon_0$, for all $\rho< R\leq R_0$, we have
$$\phi(\rho)\leq C\left[\left(\frac{\rho}{R}\right)^\beta\phi(R)+C_2\rho^\beta\right],$$
where $C$ is a positive constant depending on $C_1,\ \alpha,\ \beta$.
\end{Lem}
Here we give the H\"{o}lder continuity and interior regularity for measures $m$ in FPE \eqref{eq1.4} as follows.
\begin{Lem}\label{APP1.4}
Let $m$ be the weak solution of the FPE \eqref{eq1.4}, and $u$ be the solution of the HJE \eqref{eq1.3}. For any  $s,t\in[0,T]$, there is a constant $C$ such that
\begin{enumerate}
  \item[\rm{(i)}] $\ d_1(m_t,m_s) \leq C\|h\|_{L^{\infty}(\mathbb{R})}\big(\|\nabla_{\mathcal {X}}u\|_{L^{\infty}(U_T)}+1\big)|t-s|^{\frac{1}{2}},$
  \item[\rm{(ii)}]  $\int_{\mathbb{R}^{2}}|x|^2dm(t,x)\leq C\int_{\mathbb{R}^{2}}|x_2|^2 dm_0+C\|h\|_{L^{\infty}(\mathbb{R})}^2\big(\|\nabla_{\mathcal {X}}u\|^2_{L^{\infty}(U_T)}+1\big).$
\end{enumerate}
\end{Lem}
\begin{proof}
First, we prove {\rm{(i)}} by the same argument of Lemma 3.4 in \cite{Cardaliaguet}. By the definition \eqref{distance} of $d_1$, the law $\gamma$ of the pair $(X_t,X_s)$ belongs to $\Pi(m_t,m_s)$, so
$$ d_1(m_t,m_s) \leq \int_{\mathbb{R}^{2}\times\mathbb{R}^{2}}|x-y|d\gamma(x,y)=\mathbf{E}\big[|X_t-X_s|\big],$$
where $X_s=(X_s^1,X_s^2)$ subjects to \eqref{eq1.7}. Here we explicitly provide the calculations for the second component $X^2_{s}$, since the calculations for $X^1_{s}$ are same with $h=1$.

For instance $t<s$, since the optimal control $\alpha^*=\nabla_{\mathcal {X}}u$ is bounded, it follows from Jensen's inequality and Burkholder-Davis-Gundy inequality that
\begin{align}\label{7.31}
\mathbf{E}\left[\big|X^2_s-X_t^2\big|\right]%&\leq \mathbf{E}\left[\int_t^s\big|h^2(X^1_\tau)\partial_{x_2}u\big|d\tau\right]+\mathbf{E}\left[\left|\int^s_th(X^1_\tau)dB^2_{\tau}\right|\right]\\\nonumber
&\leq \mathbf{E}\left[\int_t^s h(X^1_\tau)\big|\nabla_{\mathcal {X}}u\big|d\tau\right]
+\bigg(\mathbf{E}\bigg[\big(\int^s_t|h(X^1_\tau)|dB_{2,\tau}\big)^2\bigg]\bigg)^{\frac{1}{2}}\\\nonumber
 &\leq C\|h\|_{L^\infty(\mathbb{R})}\|\nabla_{\mathcal {X}}u\|_{L^{\infty}(U_T)}(s-t)
  +\mathbf{E}\left[\int^s_t|h(X^1_\tau)|^2d\tau\right]^{\frac{1}{2}}\\\nonumber
  &\leq C\|h\|_{L^\infty(\mathbb{R})}\|\nabla_{\mathcal {X}}u\|_{L^{\infty}(U_T)}T^{\frac{1}{2}}(s-t)^{\frac{1}{2}}+\|h\|_{L^\infty(\mathbb{R})}(s-t)^{\frac{1}{2}}\\\nonumber
  &\leq  C\|h\|_{L^{\infty}(\mathbb{R})}\big(\|\nabla_{\mathcal {X}}u\|_{L^{\infty}(U_T)}+1\big)|t-s|^{\frac{1}{2}}.
  \end{align}

Second, we prove {\rm{(ii)}} as Lemma 3.5 in \cite{Cardaliaguet}. Similar to \eqref{7.31}, we get
\begin{align*}
  \int_{\mathbb{R}^{2}}|x_2|^2 dm(t,x)&=\mathbf{E}\left[|X^2_t|^2\right]\\
  &  \leq C\mathbf{E}\bigg[|X^2_0|^2+\left(\int_{0}^{t}h(X^1_\tau)|\nabla_{\mathcal {X}}u|d\tau\right)^2
  +\left(\int_{0}^{t}h(X^1_{\tau})dB^2_{\tau}\right)^2\bigg]\\
  &\leq C\int_{\mathbb{R}^{2}}|x_2|^2 dm_0+C\|h\|_{L^{\infty}(\mathbb{R})}^2\left(\|\nabla_{\mathcal {X}}u\|^2_{L^{\infty}(U_T)}t^2+t\right),
  \end{align*}
  and the result follows.
  \end{proof}

Finally, we recall a version of Schauder fixed point Theorem.
\begin{Lem}\label{APP1.33}{\rm{(Theorem 9.5 in \cite{An})}}
Let $B$ be a close bounded convex subset of a normed space $\mathcal {B}$. If $f:B\to\mathcal {B}$ is a compact map such that $f(B)\subseteq B$, then there exists an $x\in B$  such that $f(x)=x$.
\end{Lem}
\section{THE EXISTENCE AND UNIQUENESS OF HJE}
In this section, first we prove the WMP given in Lemma \ref{WMP4}. Then we obtain the Schauder prior estimates given in Lemma \ref{1.1}. These results are used to prove the existence, uniqueness and regularities of HJE given in Theorem \ref{1.2}.
\subsection{The proof of Lemma \ref{WMP4}}
Let $V$ be a bound set of $\mathbb{R}^2$, and $V_T:=(0,T)\times V$. Denote by
\begin{equation}\label{10.14}
  \partial_p V_T:=\big\{\{t=0\}\times V\big\}\bigcup\big\{(0,T)\times\partial V\big\},
\end{equation}
 the parabolic boundary of $V_T$. Now we prove WMP with weak regularity on $V_T$.
\begin{Prop}\label{WMP2}{\rm{(WMP on $V_T$)}}
Assume $u\in C_{\mathcal {X}}^{2}(V_T)\cap C(\bar{V}_T)$, for the heat type operator $\mathcal {H}$ defined in \eqref{001} with $b_i=0$, and $c\geq0$, we have
\begin{itemize}
  \item [\rm{(i)}] If $\mathcal {H}u \leq 0$ in $V_T$, then $\max_{\bar{V}_T}u\leq\max_{\partial_p V_T}u^+$.
  \item [\rm{(ii)}] If $\mathcal {H}u \geq 0$ in $V_T$, then $\min_{\bar{V}_T}u\geq-\max_{\partial_p V_T}u^-$.
\end{itemize}
In particular, if $\mathcal {H}u =0$ in $V_T$, then  $\max_{\bar{V}_T}|u|=\max_{\partial_p V_T}|u|$.
\end{Prop}
\begin{proof}
First, we prove assertion {\rm{(i)}}. In the case of
\begin{equation}\label{7.22}
  \mathcal {H}u<0,\ \text{in}\ V_T,
\end{equation}
assume that there exists a point $(t_0,x_0)\in V_T$  with
$u(t_0,x_0)=\max_{\bar{V}_T}u.$
Now at this maximum point $(t_0,x_0)$, by the proof of Theorem 4.1 in \cite{heat}, we have
\begin{equation}\label{7.1}
 -\sum^{q}_{i,j=1}a_{ij}(t_0,x_0)X_iX_ju(t_0,x_0)\geq0.
\end{equation}
Since $c(t_0,x_0)u(t_0,x_0)\geq0$, and $\partial_t u(t_0,x_0)\geq0$, we have
$$\partial_t u(t_0,x_0)-\sum^{q}_{i,j=1}a_{ij}(t_0,x_0)X_iX_ju(t_0,x_0)+c(t_0,x_0)u(t_0,x_0)\geq0.$$
So \eqref{7.22} is incompatible, and we have a contradiction at $(t_0,x_0)$.

Then, in the general case that $\mathcal {H}u\leq0$ in $V_T$. Let $u^{\varepsilon}(t,x)=u(t,x)-\varepsilon t$, we get
$$\mathcal {H}u^{\varepsilon}(t,x)=\mathcal {H}u(t,x)-\varepsilon(1+ct)<0,\ \text{in}\ V_T.$$
Furthermore if $u$ attains a positive maximum at some point in $V_T$, then $u^{\varepsilon}$ also attains a positive maximum at some point belonging to $V_T$, provided $\varepsilon>0$ is small enough. But then, as in the previous proof, we obtain a contradiction. Hence assertion {\rm{(i)}} is valid. Since $-u$ is a subsolution whenever $u$ is a supersolution, assertion {\rm{(ii)}} follows. Combined with assertions {\rm{(i)}} and {\rm{(ii)}}, the result follows.
\end{proof}

%\begin{Cor}\label{WMP3}{\rm{(WMP for bounded $c(t,x)$)}}
%Assume $u\in C_{\mathcal {X}}^{2}(V_T)\cap C(\bar{V}_T)$, for the heat type operator $\mathcal {H}$ defined in \eqref{001} with $b_i\equiv0$, and $c$ is bounded, we have

%\rm{(i)}If $\mathcal {H}u \leq 0$ in $V_T$, then $\max_{\bar{V}_T}u\leq\max_{\partial V_T}u^+$.

%\rm{(ii)}If $\mathcal {H}u \geq 0$ in $V_T$, then $\min_{\bar{V}_T}u\geq-\max_{\partial V_T}u^-$.

%In particular, if $\mathcal {H}u =0$ in $V_T$, then  $\max_{\bar{V}_T}|u|=\max_{\partial V_T}|u|$.
%\end{Cor}
 Next, we extend Proposition \ref{WMP2} from $V_T$ to $U_{T}$ with bounded $c(t,x)$.
\begin{proof}[\bf{Proof of Lemma \ref{WMP4}}]
The main idea of the proof is inspired by Theorem 6 in Chapter 2 of \cite{Evans}. It is not restrictive to assume that $c(t,x)\geq0$. If not the case, we set $\tilde{u}(t,x)=e^{-\lambda t}u(t,x)$, and $\tilde{u}$ satisfies \eqref{12} with $\tilde{c}(t,x)=\lambda+c(t,x)\geq0$, provided $\lambda$ is suitably chosen.
Now we assume
\begin{equation}\label{25}
4aT<1,
\end{equation}
in which case $4a(T+\varepsilon)<1$, for some $\varepsilon>0$. Fixed $y\in \mathbb{R}^2$, $\mu>0$,  we define
\begin{equation}\label{35}
 v(t,x):=u(t,x)-\frac{\mu}{(T+\varepsilon-t)^{\frac{Q}{2}}}e^{\frac{|x-y|^2_{\mathcal {X}}}{4(T+\varepsilon-t)}},
 \end{equation}
for any $(t,x)\in U_T$, where $Q$ is the homogenous dimension given in \eqref{2.24}, and $|\cdot|_{\mathcal {X}}$ is a weighted distance, denoted as follows. For any $x,y\in \mathbb{R}^2$, we define
$$|x-y|^2_{\mathcal {X}}:=(x_1-y_1)^2+\frac{(x_2-y_2)^2}{\|h\|^2_{L^{\infty}(\mathbb{R})}}.$$
A direct calculation shows for any $(t,x)\in U_T$,
\begin{align*}
  &\quad\partial_t v-\Delta_{\mathcal {X}}v+cv\\
  &=\partial_t u-\Delta_{\mathcal {X}}u+cu-\frac{\mu e^{\frac{|x-y|^2_{\mathcal {X}}}{4(T+\varepsilon-t)}}}{2(T+\varepsilon-t)^{\frac{Q}{2}+1}}  \left(Q-1-\frac{h^2(x_1)}{\|h\|^2_{L^{\infty}(\mathbb{R})}}+2c(T+\varepsilon-t)\right)\\
    &\quad-\frac{\mu e^{\frac{|x-y|^2_{\mathcal {X}}}{4(T+\varepsilon-t)}} }{4(T+\varepsilon-t)^{\frac{Q}{2}+2}} \left(1-\frac{h^2(x_1)}{\|h\|^2_{L^{\infty}(\mathbb{R})}}\right)\frac{(x_2-y_2)^2}{\|h\|^2_{L^{\infty}(\mathbb{R})}}\leq0,
  \end{align*}
which implies that $v$ is a subsolution on $U_T$.
Fixed $R>0$, and set $Q_T:=(0,T)\times B_R(y)$, the boundary $\partial_p Q_T$ defined in \eqref{10.14}.
According to Proposition \ref{WMP2}, we get
\begin{equation}\label{27}
  \max_{\bar{Q}_T}v\leq\max_{\partial_p Q_T}v^+.
\end{equation}
For any $x\in B_R(y)$, we have
\begin{equation}\label{28}
  v(0,x)=u(0,x)-\frac{\mu}{(T+\varepsilon)^{\frac{Q}{2}}}e^{\frac{|x-y|^2_{\mathcal {X}}}{4(T+\varepsilon)}}\leq u(0,x)=g(x).
\end{equation}
For any $(t,x)\in(0,T)\times\partial B_R(y)$, i.e. $d_{C}(x,y)=R$. Combined with the growth estimate \eqref{2.11}, we get
\begin{align*}
v(t,x)% &= u(t,x)-\frac{\mu}{(T+\varepsilon-t)^{\frac{M}{2}}}e^{\frac{r^2}{4(T+\varepsilon-t)}} \\
   &\leq Ce^{a(d_{C}(x))^2}-\frac{\mu}{(T+\varepsilon-t)^{\frac{Q}{2}}}e^{\frac{R^2}{4(T+\varepsilon-t)}} \\
   &\leq Ce^{a(d_{C}(y)+R)^2}-\frac{\mu}{(T+\varepsilon)^{\frac{Q}{2}}}e^{\frac{R^2}{4(T+\varepsilon)}}.
 \end{align*}
Set $\gamma=\frac{1}{4(T+\varepsilon)}-a>0$. Thus we find
\begin{equation}\label{29}
  v(t,x)\leq Ce^{a(d_{C}(y)+R)^2}-\mu(4(a+\gamma))^{\frac{Q}{2}}e^{(a+\gamma)R^{2}}\leq\sup_{\mathbb{R}^2} g^+,
\end{equation}
for $R$ selected sufficiently large. Thus inequalities \eqref{27}-\eqref{29} imply
$$\max_{\bar{Q}_T}v\leq\max_{\partial_p Q_T}v^+\leq\sup_{\mathbb{R}^2} g^+,$$
for all $y\in\mathbb{R}^2$, $0\leq t\leq T$. Furthermore, let $\mu\rightarrow0$, we have
$$u(t,x)\leq \sup_{\mathbb{R}^2} g^+.$$

In the general case, when \eqref{25} fails, we repeatedly apply the result above on the time intervals $[0,T_1]$, $[T_1,T_2]$, ect, for $T_1=\frac{1}{8a}$. Then the result follows.
\end{proof}
By Lemma \ref{WMP4}, we can obtain the following comparison principle.
\begin{Prop}\label{CP}{\rm{(Comparison Principle)}}
Given vector fields $\mathcal {X}$ defined as \eqref{eq1.60}, for the operator $\mathcal {H}$ is given by \eqref{001} with
$A=\begin{pmatrix}
1&0\\0&1
\end{pmatrix}, $
$b_i(t,x)=0$ and $c(t,x)$ is bounded. If $v,w\in C_{\mathcal {X}}^2(U_{T})$, and
$$\mathcal {H}v(t,x)\leq\mathcal {H}w(t,x),\ \text{and}\ \ v(0,x)\leq w(0,x),$$
then $v(t,x)\leq w(t,x)$, for all $(t,x)\in U_{T}$.
\end{Prop}
\begin{proof}
Take $u=v-w$ in Lemma \ref{WMP4}, then we have $u\leq 0$ in $U_{T}$, and the result is completed.
\end{proof}
By properly selecting the functions in Proposition \ref{CP}, we can obtain the prior estimate of the solution.
\begin{Prop}\label{PE}{\rm{(Prior Estimate)}}
Under the same assumptions of Proposition  \ref{CP}, if $u\in C_{\mathcal {X}}^2(U_{T})$, and $\mathcal {H}u=f$ in $U_{T}$, then
\begin{equation}\label{3.01}
\sup_{U_T}|u|\leq\sup_{x\in\mathbb{R}^2}|u(0,x)|+T\sup_{U_T}|f|.
\end{equation}
\end{Prop}
\begin{proof}
As the proof of Lemma \ref{WMP4}, it is enough to prove for $c(t,x)\geq0$, and bounded $f$. Set
$$w(t,x)=\sup_{x\in\mathbb{R}^2}|u(0,x)|+t\sup_{U_T}|f(t,x)|.$$
Then $w\in C^2_{\mathcal {X}}(U_{T})\cap C(\bar{U}_{T})$. For any $(t,x)\in U_{T}$, we have
$$\mathcal {H}w(t,x)=\sup_{U_T}|f(t,x)|+c(t,x)w(t,x)\geq\sup_{U_T}|f(t,x)|.$$
Similarly, $\mathcal {H}(-w(t,x))\leq - \sup_{U_T}|f(t,x)|$. Then for any $(t,x)\in U_T, $ we have
$$\mathcal {H}(-w(t,x))\leq\mathcal {H}u(t,x)\leq\mathcal {H}w(t,x).$$
Consider that for any $x\in \mathbb{R}^2$,
$$-w(0,x)\leq u(0,x)\leq w(0,x),$$
the inequality \eqref{3.01} is valid by the Proposition \ref{CP}.
\end{proof}

\subsection{The proof of Lemma \ref{1.1}}
In this subsection, we shall prove Lemma \ref{1.1} for the degenerate operator $\mathcal {H}$ defined in \eqref{001}. The proof of the global Schauder estimates is a blend of the localizing technique used by Theorem 8.11.1 in \cite{9}, and the method of freezing the coefficients.

Without loss of generality, assume that $\mathcal {N}(h)=\{0\}$. Set
\begin{equation}\label{Z}
  \mathcal {Z}:=\{x=(x_1,x_2)\in\mathbb{R}^2|\ x_1=0,x_2\in\mathbb{R}\}.
\end{equation}
For any $x_0\in\mathcal {Z}$, set $$U_{R}^{0}=U_{R}^{0}(x_0)=(0,R^2)\times B_{R}(x_0).$$
 Denote  $\nabla_{\mathcal {X}}^2$ the Hessian matrix, $\nabla_{\mathcal {X}}^k$  $k$-tensor product, $D^{k}_{\mathcal {X}}:=X^{I}\  \text{with} \ |I|=k.$
Denote $\varrho$ the radios of the convergence of the function $h$ at $x_1=0$.
By Taylor expansion for any $x_1\in (-\varrho,\varrho)$, we have
$$h(x_1)=x_1^\kappa\left(\frac{h^{(\kappa)}(0)}{\kappa!}+o(x_1)\right)=:x_1^\kappa\varsigma(x_1).$$
Then $\varsigma(x_1)\neq0$, and $\varsigma(x_1)\in C^{\infty}(\mathbb{R})$ by assumptions {\bf{(H1)}} and {\bf{(H2)}}. Furthermore, the operator $\Delta_{\mathcal {X}}$ can be rewritten as
$$\Delta_{\mathcal {X}}=\partial^2_{x_1}+\varsigma^2(x_1)(x_1^\kappa\partial_{x_2})^2.$$
Regarding $\varsigma(x_1)$ as the variable coefficient, we first consider the following equation
\begin{equation}\label{7.2}
\left\{
  \begin{array}{ll}
    \partial_t u-\Delta_{\mathcal {Y}}u=f,\quad (t,x)\in U_{R_0}^{0},\\
     u(0,x)=0, \quad\qquad x\in B_{R_0},
  \end{array}
\right.
\end{equation}
 for any $0<R_0<\varrho$, $x_0\in\mathcal {Z}$, where the vector fields
\begin{equation}\label{8.13}
 \mathcal {Y}:=\{Y_1,Y_2\},\ \text{with}\ Y_1:=\partial_{x_1}, Y_2:=x_1^\kappa\partial_{x_2}.
\end{equation}

\begin{Rem}\label{10.16}
Note that the operator $\partial_t-\Delta_{\mathcal {Y}}$ is $2$-homogeneous and hypoelliptic, by H\"{o}rmander theorem. Moreover,
if $u$ is a solution of the equation \eqref{7.2} with $f=0$,
then {\bf{(H2)}} ensures the homogeneity of $u$ with index $2$, it implies that
 \begin{equation}\label{ho1}
  u(t,x)=u(t,x_1,x_2)=\lambda^{-2}u(\lambda^2t,\lambda x_1,\lambda^{\kappa+1}x_2).
\end{equation}
 By the way, the relation \eqref{ho1} also implies \eqref{2.24}. More details for H\"{o}rmander operators on homogeneous groups see in \cite{heat} and \cite{Bra}.
%It is emphasized that the rescaling technology used as below depends on the homogeneity of the solution.
% Then  for the equation \eqref{10.08} follows from an easy perturbation argument using a partition of unity.
\end{Rem}

First, we show the Caccioppoli-type inequality on weighted Sobolev spaces.
\begin{Prop}\label{7.3.1}{\rm{(Caccioppoli-type inequality)}}
Let $u\in H^{1}_{\mathcal {Y}}(U_{T})$ be a weak solution of \eqref{7.2} with $f=0$, and $0<R\leq R_0<\varrho$. For any $x_0\in\mathcal {Z}$, such that $U_{R}^{0}\subset U_T$, and for any integer $k>0$, we have
\begin{equation}\label{03.04}
 \sup_{t\in\big(0,\frac{R^2}{2^{2k}}\big)}\int_{B_{{R}/{2^{k}}}}\big|\nabla_{\mathcal {Y}}^{k-1}u\big|^2dx+\|u\|^2_{H_{\mathcal {Y}}^k\big({U_{R/2^k}^{0}}\big)}\leq \frac{C}{R^{2k}}\|u\|^2_{L^2(U^0_{R})}.
\end{equation}
Furthermore, we have
\begin{equation}\label{03.4}
 \iint_{U_{R/2^{k}}^{0}}\big|\nabla_{\mathcal {Y}}^{k}u\big|^2dxdt\leq \frac{C}{R^{2k}}\iint_{U_{R}^{0}}u^2dxdt.
\end{equation}
%where the constant $C$ is independent of $x^i$, $R$ and $u$.
\end{Prop}
\begin{proof}
The idea of the proof is inspired of Proposition 3.1 in \cite{Xu-56}. The function $u$ is $C^\infty$ since $\mathcal {H}$ is hypoelliptic. We shall  write $B_R$ instead of $B_{R}(x_0)$, and prove \eqref{03.04} by induction on $k$. First, we give the proof of \eqref{03.04} with $k=1$. Let the cutoff function $\eta\in C^{\infty}_{0}(B_{R})$, satisfying
\begin{equation}\label{7.022}
  0\leq\eta\leq1,\ \eta\equiv1\ \text{on}\ B_{{R}/{2}},\ \text{and}\ |\nabla_{\mathcal {Y}}\eta|\leq CR^{-1}.
\end{equation}
 Multiply both sides of equation \eqref{7.2} by $\eta^2u$, for all $s\in\big(0,\frac{R^2}{4}\big)$, and integrate on $(0,s)\times B_{R}$.  Using integration by parts, we get
\begin{align}\label{3.10}
  \int_0^s\int_{B_{R}}\eta^2u\partial_t udxdt &=\int_0^s\int_{B_{R}}\eta^2u\Delta_{\mathcal {Y}}udxdt\\\nonumber
  &=-\int_0^s\int_{B_{R}}\eta^2|\nabla_{\mathcal {Y}}u|^2+2\eta u\nabla_{\mathcal {Y}}u \nabla_{\mathcal {Y}}\eta dxdt.
\end{align}
For the left hand of \eqref{3.10}, since $u\big|_{t=0}=0$, we have
\begin{equation}\label{3.11}
  \int_0^s\int_{B_{R}}\eta^2u\partial_t udxdt=\frac{1}{2}\int_{B_{R}}(\eta u)^2\bigg|_{t=s}dx.
\end{equation}
%We deduce from \eqref{3.10} and \eqref{3.11} that
%\begin{equation}\label{8.19}
 % \frac{1}{2}\int_{B_{R}}(\eta u)^2\bigg|_{t=s}dx+\int_0^s\int_{B_{R}}\eta^2|\nabla_{\mathcal {Y}}u|^2dxdt=-2\int_0^s\int_{B_{R}}\eta u\nabla_{\mathcal {Y}}u \nabla_{\mathcal {Y}}\eta dxdt.
%\end{equation}
It follows from \eqref{3.10}, \eqref{3.11}, the $\varepsilon$-Cauchy inequality and \eqref{7.022} that
\begin{align}\label{8.18}
&\quad\frac{1}{2}\int_{B_{R}}(\eta u)^2\bigg|_{t=s}dx+\int_0^s\int_{B_{R}}\eta^2|\nabla_{\mathcal {Y}}u|^2dxdt\\\nonumber
&=2\int_0^s\int_{B_{R}}(\eta\nabla_{\mathcal {Y}}u) (u \nabla_{\mathcal {Y}}\eta) dxdt\\\nonumber
   & \leq\varepsilon \int_0^s\int_{B_{R}}\eta^2|\nabla_{\mathcal {Y}}u|^2dxdt+C_{\varepsilon}\int_0^s\int_{B_{{R}}}|\nabla_{\mathcal {Y}}\eta|^2 u^2 dxdt\\\nonumber
  &\leq \varepsilon \int_0^s\int_{B_{R}}\eta^2|\nabla_{\mathcal {Y}}u|^2dxdt+\frac{C_{\varepsilon}}{R^{2}}\int_0^s\int_{B_{{R}}}u^2dxdt,
\end{align}
for any $s\in\big(0,\frac{{R}^2}{4}\big)$. Taking $\varepsilon$ small enough, it follows from \eqref{8.18} and \eqref{7.022} that
%$$\int_{B_{R}}(\eta u)^2\bigg{|}_{t=s}dx+\int_0^s\int_{B_{R}}|\eta \nabla_{\mathcal {Y}}u|^2dxdt\leq \frac{C}{R^{2}_0}\int_0^s\int_{B_{{R}}}u^2dxdt,$$
% Since $\eta=1$ on $B_\frac{{R}}{2}$, we have
\begin{equation}\label{8.22}
 \sup_{t\in\big(0,\frac{R^2}{4}\big)}\int_{B_{{R}/{2}}}|u|^2dx+\|u\|^2_{H_{\mathcal {Y}}^1\big({U_{R/2}^{0}}\big)}\leq \frac{C}{R^{2}}\iint_{U_{R}^{0}}u^2dxdt,
\end{equation}
it implies that \eqref{03.4} is valid for $k=1$.

Assume now \eqref{03.4} true up to the order $k-1$, that is for any integer $2\leq i\leq k-1$,
\begin{equation}\label{8.17}
   \sup_{t\in\big(0,{R^2}/{2^{2i}}\big)}\int_{B_{R/2^{i}}}\big|\nabla_{\mathcal {Y}}^{i-1}u\big|^2dx+\|u\|^2_{H_{\mathcal {Y}}^i\big({U_{R/2^i}^{0}}\big)}\leq \frac{C}{R^{2i}}\iint_{U_{R}^{0}}u^2dxdt.
\end{equation}
Applying $D_{\mathcal {Y}}^{k-1}$ to the equation \eqref{7.2} with $f=0$, we have
\begin{equation}\label{3.18}
  \partial_t D_{\mathcal {Y}}^{k-1}u -D_{\mathcal {Y}}^{k-1}\Delta_{\mathcal {Y}}u=0.
\end{equation}

 Similar to \eqref{3.10}-\eqref{8.18}, multiply both sides of the equation \eqref{3.18} by $\eta^2D_{\mathcal {Y}}^{k-1}u$ and integrate on $(0,s)\times B_{{R}/{2^{k-1}}}$, where the cutoff function $\eta\in C^{\infty}_{0}(B_{{R}/{2^{k-1}}})$ satisfies
\begin{equation}\label{8.16}
  0\leq\eta\leq1,\ \eta\equiv1\ \text{on}\ B_{{{R}}/{2^k}}, \ \text{and}\ |\nabla_{\mathcal {Y}}^{k}\eta|\leq CR^{-k}.
\end{equation}
Note that $|\Delta_{\mathcal {Y}}u|\leq|\nabla^2_\mathcal {Y}u|$,
it follows from integration by parts, Cauchy inequality and \eqref{8.16} that
\begin{align}\label{3.21}
  &\quad\frac{1}{2}\int_{B_{{R}/{2^{k-1}}}}\eta^2 |D_{\mathcal {Y}}^{k-1}u|^2\bigg|_{t=s} dx\\\nonumber
   & =-\int_0^s\int_{B_{{R}/{2^{k-1}}}}\left(\eta^2D_{\mathcal {Y}}^{k}u+2\eta D_{\mathcal {Y}}\eta D_{\mathcal {Y}}^{k-1}u\right)D_{\mathcal {Y}}^{k-2}\Delta_{\mathcal {Y}}udxdt \\\nonumber
   & \leq C\int_0^s\int_{B_{{R}/{2^{k-1}}}}\eta^2|\nabla^{k}_{\mathcal {Y}}u|^2dxdt+C\int_0^s\int_{B_{{R}/{2^{k-1}}}}|D_{\mathcal {Y}}\eta|^2| \nabla^{k-1}_\mathcal {Y}u|^2dxdt\\\nonumber
   & \leq C\int_0^s\int_{B_{{R}/{2^{k-1}}}}\eta^2|\nabla^{k}_{\mathcal {Y}}u|^2dxdt+\frac{C}{R^{2}}\int_0^s\int_{B_{{R}/{2^{k-1}}}}|\nabla^{k-1}_\mathcal {Y}u|^2dxdt\\\nonumber
   &\leq C\|\eta u\|^2_{H^{k}_{\mathcal {Y}}\big(U^0_{{R}/{2^{k-1}}}\big)}+\frac{C}{R^{2}}\| u\|^2_{H^{k-1}_{\mathcal {Y}}\big(U^0_{{R}/{2^{k-1}}}\big)},
\end{align}

Furthermore, using the maximum regularity for the operator $L=\partial_t-\Delta_\mathcal {Y}$, which proved by Theorem 18 of \cite{RS}, we get
\begin{align}\label{10.04}
\|\eta u\|_{H^{k}_{\mathcal {Y}}\big(U^0_{{R}/{2^{k-1}}}\big)}\leq C\|L(\eta u)\|_{H^{k-2}_{\mathcal {Y}}\big(U^0_{{R}/{2^{k-1}}}\big)}+C\|\eta u\|_{L^2\big(U^0_{{R}/{2^{k-1}}}\big)}.
\end{align}
Since $Lu=0$, and $\partial_t\eta\equiv0$, we have
\begin{equation}\label{8.15}
L(\eta u)=-\sum^{2}_{j=1}\big(uX_j^2\eta+2X_juX_j\eta\big).
\end{equation}
 It follows from the estimates \eqref{10.04}, \eqref{8.15} and the induction \eqref{8.17} that
\begin{align}\label{8.20}
  \|L(\eta u)\|_{H^{k-2}_{\mathcal {Y}}\big(U^0_{{R}/{2^{k-1}}}\big)} %&\leq C_k\left(R^{-2}\|u\|_{H^{k-2}_{\mathcal {Y}}(U^0_{R/2^{k-1}})}+R^{-1}\|Xu\|_{H^{k-2}_{\mathcal {Y}}(U^0_{R/2^{k-1}})}\right) \\
  &\leq\frac{C}{R^{2}}\|u\|_{H^{k-2}_{\mathcal {Y}}\big(U^0_{{R}/{2^{k-1}}}\big)}+\frac{C}{R}\|u\|_{H^{k-1}_{\mathcal {Y}}\big(U^0_{{R}/{2^{k-1}}}\big)}\\\nonumber
  &\leq \frac{C}{R^{k}}\|u\|_{L^2(U^0_{R})}.
\end{align}
It follows from \eqref{10.04} and \eqref{8.20} that
\begin{equation}\label{3.16}
  \|u\|_{H^{k}_{\mathcal {Y}}\big(U^0_{{R}/{2^{k}}}\big)}\leq\|\eta u\|_{H^{k}_{\mathcal {Y}}\big(U^0_{{R}/{2^{k-1}}}\big)}\leq \frac{C}{R^k}\|u\|_{L^2(U^0_{R})}.
\end{equation}
It follows from \eqref{3.21}, \eqref{3.16} and \eqref{8.17} with $i=k-1$ that
\begin{align}\label{8.21}
 \sup_{t\in\big(0,\frac{R^2}{2^{2k}}\big)}\int_{B_{{R}/{2^{k}}}}|\nabla^{k-1}_{\mathcal {Y}}u|^2 dx &\leq C\|\eta u\|^2_{H^{k}_{\mathcal {Y}}\big(U^0_{{R}/{2^{k-1}}}\big)}+\frac{C}{R^{2}}\| u\|^2_{H^{k-1}_{\mathcal {Y}}\big(U^0_{{R}/{2^{k-1}}}\big)}\\\nonumber
 &\leq\frac{C}{R^{2k}}\|u\|^2_{L^2(U^0_{R})}.
\end{align}
Combined with \eqref{3.16} and \eqref{8.21}, we get \eqref{03.04} for $k$, and \eqref{03.4} follows.
\end{proof}
Furthermore, we extend the maximum regularity result in \cite{RS} to inhomogeneous equation \eqref{7.2}.
\begin{Cor}\label{03.2}
Let $u\in C^{\infty}(\bar{U}_{T})$ be a weak solution of equation \eqref{7.2}, then for all $x_0\in\mathcal {Z}$, such that $U_{R}^{0}\subset U_T$, we have
\begin{equation}\label{7.2.13}
  \sup_{t\in(0,\frac{R^2}{16})}\int_{B_{{R}/{4}}}|\nabla_{\mathcal {Y}}u|^2dx+\|\nabla^2_{\mathcal {Y}}u\|_{L^2(U_{{R}/{4}}^0)}
  \leq \frac{C}{R^{4}}\|u\|^2_{L^2(U^0_{R})}+C\|f\|^2_{L^2(U^0_{R})}.
\end{equation}
%where the constant $C$ is independent of $x^i$, $R$ and $u$.
\end{Cor}
\begin{proof}
 First, let $\eta$ be the cutoff function satisfying \eqref{8.16} with $k=2$, and write $B_R$ instead of $B_{R}(x_0)$. By the same argument in \eqref{7.022}-\eqref{8.22}, multiply both sides of equation \eqref{7.2} by $\eta^2u$, and integrate on $(0,s)\times B_{R}$. Similar to \eqref{8.22},  we get
\begin{equation}\label{3.30}
\|u\|^2_{H_{\mathcal {Y}}^1\big({U_{R/2}^{0}}\big)}
  \leq \frac{C}{R^{2}}\|u\|^2_{L^2(U^0_{R})}+C\|f\|^2_{L^2(U^0_{R})}.
\end{equation}
Similar to \eqref{10.04}, by the maximum regularity,  we get
\begin{equation}\label{9.09}
  \|\nabla^2_{\mathcal {Y}}u\|_{L^2(U_{{R}/{4}}^0)}\leq \|\eta u\|^2_{H^{2}_{\mathcal {Y}}(U^0_{{R}/{2}})}
\leq C\big(\|L(\eta u)\|^2_{L^2(U^0_{{R}/{2}})}+\|\eta u\|^2_{L^2(U^0_{{R}/{2}})}\big).
\end{equation}
Since $Lu=f$, and $\partial_t\eta\equiv0$, we have $L(\eta u)=-\sum^{2}_{i=1}(uX_i^2\eta+2X_iuX_i\eta)-\eta f$, and
\begin{align}\label{10.06}
\|L(\eta u)\|^2_{L^2(U^0_{{R}/{2}})}\leq C\bigg(\frac{1}{R^{4}}\|u\|^2_{L^2(U^0_{{R}/{2}})}+\frac{1}{R^{2}}\|\nabla_{\mathcal {Y}}u\|^2_{L^2(U^0_{{R}/{2}})}+\|f\|^2_{L^2(U^0_{{R}/{2}})}\bigg).
\end{align}
It follows from \eqref{3.30}-\eqref{10.06} and \eqref{7.022} that
\begin{align}\label{3.31}
  &\quad\|\nabla^2_{\mathcal {Y}}u\|_{L^2(U_{{R}/{4}}^0)}\\\nonumber
  &\leq \frac{C}{R^{4}}\|u\|^2_{L^2(U^0_{{R}/{2}})}+C\|\eta u\|^2_{L^2(U^0_{{R}/{2}})}+\frac{C}{R^{2}}\|\nabla_{\mathcal {Y}}u\|^2_{L^2(U^0_{{R}/{2}})}+C\|f\|^2_{L^2(U^0_{{R}/{2}})}\\\nonumber
  &\leq \frac{C}{R^{4}}\|u\|^2_{L^2(U_{R}^{0})}+C\| f\|^2_{L^2(U_{R}^{0})}.
\end{align}

Second, applying the operator $D_{\mathcal {Y}}$ to the equation \eqref{7.2}, then we have
\begin{equation}\label{8.23}
\partial_tD_{\mathcal {Y}}u-D_{\mathcal {Y}}\Delta_{\mathcal {Y}}u=D_{\mathcal {Y}}f.
\end{equation}
Let $\eta$ be the cutoff function satisfying \eqref{8.16}. By the same argument \eqref{3.21}-\eqref{8.21}, multiply
both sides of \eqref{8.23} by $\eta^2D_{\mathcal {Y}}u$, and integrate on $(0,s)\times B_{{R}/{2}}$. It follows from integration by parts, Cauchy inequality and the estimate \eqref{8.16} that
\begin{align}\label{3.32}
  &\quad\int_{B_{{R}/{4}}}\eta^2 |D_{\mathcal {Y}}u|^2\bigg|_{t=s} dx\\\nonumber
   & =-2\int_0^s\int_{B_{{R}/{4}}}\left(\eta D_{\mathcal {Y}}^{2}u+2 D_{\mathcal {Y}}\eta D_{\mathcal {Y}}u\right )(\eta \Delta_{\mathcal {Y}}u)dxdt \\\nonumber
   &\quad-2\int_0^s\int_{B_{{R}/{4}}}(\eta f)\left(2 D_{\mathcal {Y}}\eta D_{\mathcal {Y}}u+\eta D_{\mathcal {Y}}^2u\right)dxdt\\\nonumber
   & \leq C\left(\int_0^s\int_{B_{{R}/{4}}}|\eta\nabla^2_{\mathcal {Y}}u|^2dxdt+\int_0^s\int_{B_{{R}/{2}}}| D_{\mathcal {Y}}\eta D_{\mathcal {Y}}u|^2+(\eta f)^2dxdt\right)\\\nonumber
   &\leq C\left(\int_0^s\int_{B_{{R}/{4}}}|\nabla^2_{\mathcal {Y}}u|^2dxdt + \frac{1}{R^{2}}\int_0^s\int_{B_{{R}/{2}}}| \nabla_\mathcal {Y}u|^2dxdt+\int_0^s\int_{B_{R}} f^2dxdt\right).
\end{align}
It follows from \eqref{3.30}, \eqref{3.31} and \eqref{3.32} that
\begin{align}\label{10.05}
\sup_{t\in\big(0,\frac{R^2}{16}\big)}\int_{B_{{R}/{4}}}|\nabla_{\mathcal {Y}}u|^2dx&\leq\sup_{t\in\big(0,\frac{R^2}{16}\big)}\int_{B_{{R}/{4}}}\eta^2 |\nabla_{\mathcal {Y}}u|^2dx\\\nonumber
&\leq \frac{C}{R^{4}}\|u\|^2_{L^2(U_{R}^{0})}+C\|f\|^2_{L^2(U_{R}^{0})}.
\end{align}
Combined with inequalities \eqref{3.31} and \eqref{10.05}, the result follows.
\end{proof}
%Taking $\rho=\frac{R}{2}$ in \eqref{7.2.13}, we obtain
%\begin{Cor}\label{7.2.4}
%Under the same assumptions of Corollary \ref{03.2}, we have
%$$\sup_{t\in(0,R^2/4)}\int_{B_{R/2}}|Xu|^2dx+\iint_{U^{0}_{R/2}}|D^2_{\mathcal {Y}}u|^2dxdt\\
%  \leq \frac{C}{R^4}\iint_{U_{R}^{0}}u^2dxdt+CR^{Q+2}\|f\|^2_{L^\infty(U^0_R)}.$$
%\end{Cor}
%In particular, let $R\equiv1$, reuse Corollary \ref{03.2}, we get
%\begin{Cor}\label{07.2.5}
%If $u\in H^{k}_{\mathcal {Y}}(U_{T})$ is a weak solution of \eqref{7.2}, $U_1^0\subset U$, and $f\equiv0$ on $U^0_1$, then for all nonnegative integer $k$, there exists a positive constant $C$ only depends on $n$ and $k$, such that
%\begin{equation}\label{7.2.15}
%  \sup_{t\in(0,1/4)}\int_{B_{1/2}}|D^k_{\mathcal {Y}}u|^2dx+\iint_{U^{0}_{1/2}}|D^{k+1}_{\mathcal {Y}}u|^2dxdt\leq C\iint_{U_{1}^{0}}u^2dxdt.
%\end{equation}
%\end{Cor}
%In the sequel, if the function $h(x_1)$ is vanishing at some $x_1^0\in\mathbb{R}$, with radius of convergence $\varrho>0$. Then we need to take $x^0=(x_1^0,x_2)$ as the center of the ball.
%To use the rescaling technology, now we claim that the solution $u$ holds the homogeneity.
It is emphasized that the following result uses the rescaling technology, which depends on the homogeneity of the solution given in Remark \ref{10.16}.
In the sequel, $x_0\in\mathcal {Z}$ is a fixed point, and $B_R$ stands for $B_{R}(x_0)$.
\begin{Cor}\label{07.2.6}
Under the same assumptions of Proposition \ref{7.3.1}, for any $0<R\leq R_0<\varrho$, there exists a positive constant $C$, such that
\begin{equation*}
  \sup_{U^0_{{R}/{2^{k+2\tilde{\kappa}+2}}}}|\nabla^k_{\mathcal{X}}u|^2\leq \frac{C}{R^{Q+2+2k}}\iint_{U_{R}^{0}}u^2dxdt,
\end{equation*}
where $\tilde{\kappa}$ is the H\"{o}rmander index given in \eqref{2.78}, and $Q$ is the homogenous dimension given in \eqref{2.24}.
\end{Cor}
\begin{proof}
First, assume $R\geq1$. By $\partial_t u=\Delta_{\mathcal {Y}}u$, for any positive integers $k,i,j$, and $2i+j\leq k$, we get
\begin{equation}\label{10.20}
 \iint_{U^{0}_{1/2^{k}}}|\partial^i_t \nabla^j_{\mathcal {Y}}u|^2dxdt
\leq C\iint_{U^{0}_{1/2^{k}}}|\nabla^{2i+j}_{\mathcal {Y}}u|^2dxdt\leq C\iint_{U^{0}_{1}}u^2dxdt,
\end{equation}
and the last inequality is derived by Proposition \ref{7.3.1} with $R=1$.

Now we claim that the following embedding holds
\begin{equation}\label{9.01}
H^{2\tilde{\kappa}+2}_{\mathcal {Y}}(U^0_1)\hookrightarrow L^{\infty}(U^0_1).
\end{equation}
If $u\in H^{2\tilde{\kappa}+2}_{\mathcal {Y}}(U^0_1)$, then $D_{\mathcal {Y}}^{2\tilde{\kappa}}\partial_t u\in L^2((0,1);B_1)$, it implies that
\begin{equation}\label{9.02}
u\in H^1((0,1);H_{\mathcal {Y}}^{2\tilde{\kappa}}(B_1)).
\end{equation}
Since $H^2(B_1)\hookrightarrow L^{\infty}(B_1)$, and the Sobolev embedding \eqref{10.19}, we have
\begin{equation}\label{9.03}
H_{\mathcal {Y}}^{2\tilde{\kappa}}(B_1)\hookrightarrow H^2(B_1)\hookrightarrow L^{\infty}(B_1).
\end{equation}
Combined with \eqref{9.02} and \eqref{9.03}, the claim \eqref{9.01} is valid.

%Hence by the embedding we have
%$$\sup_{U^0_{2^{-k}}}|D_{\mathcal {Y}}^ku|^2\leq C\|D_{\mathcal {Y}}^ku\|^2_{H^\theta_{\mathcal {Y}}(U^0_{2^{-k}})}.$$
It follows from \eqref{9.01} and \eqref{10.20} that
\begin{align}\label{10.17}
\sup_{U^0_{1/2^{k+2\tilde{\kappa}+2}}}|D_{\mathcal {Y}}^ku|^2%&\leq C\|D_{\mathcal {Y}}^ku\|^2_{H^{2\tilde{\kappa}+2}_{\mathcal {Y}}(U^0_{1/2^{k+2\tilde{\kappa}+2}})}\\\nonumber
%&
\leq C\|u\|^2_{H^{2\tilde{\kappa}+2+k}_{\mathcal {Y}}\big(U^0_{1/2^{k+2\tilde{\kappa}+2}}\big)}%\\\nonumber&
\leq C\|u\|^2_{L^2(U^0_{1})}.
\end{align}

%For any $R>0$, we use the rescaling technology. Let
%\begin{equation}\label{9.13}
%t=R^2s,\ x_1=Ry_1,\ x_2=R^{\kappa+1}y_2.
%\end{equation}
Recall $\tilde{\kappa}=\kappa+1$, \eqref{ho1} given in Remark \ref{10.16},
and Definition \ref{2.6}, let $$\gamma'(\tau)=\sum^2_{j=1}\lambda_j(\tau)Y_j(\gamma(\tau)), a.e.,\ \text{and}\ \gamma(0)=x_0,$$
 with $\sum^2_{j=1}\lambda^2_j\leq1$ a.e.. By rescaling technology, we have
\begin{align}\label{9.13}
  \sup_{(t,x)\in U^0_{{R}/{2^{k+2\tilde{\kappa}+2}}}}|D_{\mathcal {Y}}^ku(t,x)|^2
  &=  \sup_{t\in(0,{R^2}/{2^{2(k+2\tilde{\kappa}+2)}})}\sup_{\tau\in(0,{R}/{2^{k+2\tilde{\kappa}+2}})}|D_{\mathcal {Y}}^ku(t,\gamma(\tau))|^2 \\\nonumber
  &=  \sup_{s\in(0,{1/2^{2(k+2\tilde{\kappa}+2)}})}\sup_{\tilde{\tau}\in(0,{1/2^{k+2\tilde{\kappa}+2}})}|D_{\mathcal {Y}}^ku(R^2s,\gamma(R\tilde{\tau}))|^2 \\\nonumber
  &=\sup_{(s,y_1,y_2)\in U^0_{{1/2^{k+2\tilde{\kappa}+2}}}}\frac{1}{R^{2k}}|D_{\mathcal {Y}}^ku(R^2s,Ry_1,R^{\kappa+1}y_2)|^2\\\nonumber
  &  =\frac{1}{R^{2k-4}}\sup_{(s,y)\in U^0_{1/2^{k+2\tilde{\kappa}+2}}}|D_{\mathcal {Y}}^ku(s,y)|^2.
  \end{align}

It follows from \eqref{9.13} and \eqref{10.17} that
\begin{align*}
  \sup_{(t,x)\in U^0_{{R}/{2^{k+2\tilde{\kappa}+2}}}}|D_{\mathcal {Y}}^ku(t,x)|^2
  &=\frac{1}{R^{2k-4}}\sup_{(s,y)\in U^0_{1/2^{k+2\tilde{\kappa}+2}}}|D_{\mathcal {Y}}^ku(s,y)|^2\\
 &\leq \frac{C}{R^{2k-4}}\iint_{U^{0}_{1}}u^2(s,y)dyds\\
  %&= \frac{C}{R^{Q+2k-2}}\iint_{U^{0}_{1}}u^2(s,y)d(Ry)d(R^2s)\\
  %&= \frac{C}{R^{Q+2k-2}}\iint_{U^{0}_{R}}u^2(R^{-2}t,R^{-1}y)dxdt\\
  &=\frac{C}{R^{Q+2k+2}}\iint_{U^{0}_{R}}u^2(t,x)dxdt,
\end{align*}
the last equality can be treated as \eqref{9.13} by rescaling technology and \eqref{ho1}.

For the case $0<R<1$, there exists a integer $l$ such that $R\geq 2^{-l}$. Similar to \eqref{10.20}, we obtain
\begin{equation*}
\|u\|_{H^{k}_{\mathcal {Y}}\big(U^0_{1/2^{l+k}}\big)}\leq C_l\|u\|_{L^2\big(U^0_{1/2^{l}}\big)}.
\end{equation*}
Let $t=(2^lR)^2s, x=2^lR$, using the rescaling technology again, the result follows.

\end{proof}
\begin{Prop}\label{7.2.8}
Under the same assumptions of Proposition \ref{7.3.1}, for any  $0<\rho\leq R\leq R_0<\varrho$, there exists a positive constant $C$, such that
\begin{equation}\label{8.4}
\iint_{U^0_\rho}|\nabla^{2}_{\mathcal {Y}}u|^2dxdt\leq C\left(\frac{\rho}{R}\right)^{Q+4}\iint_{U^0_R}|\nabla^{2}_{\mathcal {Y}}u|^2dxdt.
\end{equation}
\end{Prop}
\begin{proof}
 For the case $0<\rho<\frac{R}{2^{k+2\tilde{\kappa}+2}}$, by Corollary \ref{07.2.6}, we have
\begin{align}\label{7.2.7}
  \iint_{U^0_\rho}|\nabla^k_{\mathcal {Y}}u|^2dxdt &\leq |U^0_\rho|\sup_{U^0_\rho}|\nabla^k_{\mathcal {Y}}u|^2\\\nonumber
  &\leq C\rho^{Q+2}\sup_{U^0_{{R}/{2^{k+2\tilde{\kappa}+2}}}}|\nabla^k_{\mathcal {Y}}u|^2 \\\nonumber
  &\leq \frac{C}{R^{2k}}\left(\frac{\rho}{R}\right)^{Q+2}\iint_{U^0_R}u^2dxdt.
 \end{align}
 %Consider that $D^{2}_{\mathcal {Y}}u$ satisfies  $$\partial_t D^{2}_{\mathcal {Y}}u-D^{2}_{\mathcal {Y}}\Delta_{\mathcal {Y}}u=0,\ \text{on}\ U^0_R.$$
 Since $u\big|_{t=0}=0$, it follows from Leibniz formula and the H\"{o}lder inequality that
 \begin{align*}
   \int^{R^2}_0|u|^2ds=\int^{R^2}_0\left(\int^s_0\partial_tu(\xi)d\xi\right)^2ds &\leq R^2\int^{R^2}_0\int^s_0|\partial_tu(\xi)|^2d\xi ds\\\nonumber
    &\leq CR^4\int^{R^2}_0|\partial_tu(\xi)|^2d\xi.
 \end{align*}
 Since $\partial_tu=\Delta_{\mathcal {Y}}u$, and $|\Delta_{\mathcal {Y}}u|\leq|\nabla^2_\mathcal {Y}u|$, we have
  \begin{align}\label{8.26}
   \iint_{U^0_R}|u|^2dsdx\leq CR^4\iint_{U^0_R}|\partial_tu(\xi)|^2d\xi dx\leq CR^4\iint_{U^0_R}|\nabla^2_\mathcal {Y}u|^2dxdt.
 \end{align}
Similar to \eqref{8.26}, since $D^{2}_{\mathcal {Y}}u\big|_{t=0}=0$, we have
\begin{equation}\label{8.24}
\iint_{U^0_{\rho}}|\nabla^{2}_{\mathcal {Y}}u|^2dxdt \leq C\rho^4\iint_{U^0_{\rho}}|\partial_t \nabla^{2}_{\mathcal {Y}}u|^2dxdt\leq C\rho^4\iint_{U^0_{\rho}}|\nabla^{4}_{\mathcal {Y}}u|^2dxdt.
\end{equation}
  It follows from \eqref{7.2.7}-\eqref{8.24} that
\begin{align*}
  \iint_{U^0_{\rho}}|\nabla^{2}_{\mathcal {Y}}u|^2dxdt & \leq C\rho^4\iint_{U^0_{\rho}}|\nabla^{4}_{\mathcal {Y}}u|^2dxdt\\
  &\leq C\rho^4\left(\frac{\rho}{R}\right)^{Q+2}\frac{1}{R^{8}}\iint_{U^0_{R}}u^2dxdt \\
   & \leq C\left(\frac{\rho}{R}\right)^{Q+6}\iint_{U^0_{R}}|\nabla^2_{\mathcal {Y}}u|^2dxdt.
\end{align*}

For the case $\frac{R}{2^{k+2\tilde{\kappa}+2}}\leq\rho\leq R$, take $C=2^{(k+2\tilde{\kappa}+2)(Q+4)}$, and the result follows.
\end{proof}
With the aid of the above results one can prove the following.
\begin{Prop}\label{7.2.9}
If $u\in C^\infty(\bar{U}_{T})$ is a weak solution of equation \eqref{7.2}, $U^0_{R}\subset U_T$, then for any $0<\rho\leq R\leq R_0<\varrho$, we have
\begin{align*}
\frac{1}{\rho^{Q+2+2\alpha}}\|\nabla^2_{\mathcal {Y}} u\|^2_{L^2(U^0_{\rho})} \leq
\frac{C}{R^{Q+2+2\alpha}}\|\nabla^2_{\mathcal {Y}} u\|^2_{L^2(U^0_{R})}+\frac{C}{R^{2\alpha}}\|f\|^2_{L^{\infty}(U^0_{R})}+C[f]^2_{C^{0,\alpha}_{\mathcal {Y}} (U^0_{R})}.
\end{align*}
\end{Prop}
\begin{proof}
Let us decompose $u=u_1+u_2$, where
\begin{equation}\label{007.1}
\left\{
 \begin{array}{ll}
  \partial_t u_1-\Delta_{\mathcal {Y}}u_1=0,\quad\ (t,x)\in U^0_{R},\\
  u_1\big|_{\partial_p U^0_R}=u\big|_{\partial_p U^0_R},
 \end{array}
\right.
\end{equation}
and
\begin{equation}\label{007.2}
\left\{
 \begin{array}{ll}
  \partial_t u_2-\Delta_{\mathcal {Y}}u_2=f,\quad \ (t,x)\in U^0_{R},\\
  u_2\big|_{\partial_p U^0_R}=0.
 \end{array}
\right.
\end{equation}

For the homogeneous equation \eqref{007.1}, by Proposition \ref{7.2.8}, we get
\begin{equation}\label{8.1}
\iint_{U^0_\rho}|\nabla^2_{\mathcal {Y}}u_1|^2dxdt\leq C\left(\frac{\rho}{R}\right)^{Q+4}\iint_{U^0_R}|\nabla^2_{\mathcal {Y}}u_1|^2dxdt.
\end{equation}

For the inhomogeneous equation \eqref{007.2}, multiply both sides of \eqref{007.2} by $\partial_tu_2$, and integrate on $U_{R}^{0}$. By $\varepsilon$-Cauchy inequality, we have
\begin{align}\label{8.29}
  \iint_{U_R^0}|\partial_t u_2|^2dxdt-\iint_{U_R^0}\partial_tu_2\Delta_{\mathcal {Y}}u_2dxdt&=\iint_{U_R^0}f\partial_tu_2dxdt\\\nonumber
  &\leq C_\varepsilon\iint_{U_R^0}f^2dxdt+\varepsilon\iint_{U_R^0}|\partial_tu_2|^2dxdt.
\end{align}
By using integration by parts, since $u_2\big|_{\partial_p U^0_R}=0$ and $D_{\mathcal {Y}}u_2\big|_{t=0}=0$, we get
\begin{align}\label{8.30}
  \iint_{U_R^0}\partial_tu_2\Delta_{\mathcal {Y}}u_2dxdt &=-\iint_{U_R^0}\nabla_{\mathcal {Y}}u_2\partial_t\nabla_{\mathcal {Y}}u_2dxdt \\\nonumber
  & =-\frac{1}{2}\iint_{U_R^0}\partial_t|\nabla_{\mathcal {Y}}u_2|^2dxdt\\\nonumber
  &=-\frac{1}{2}\int_{B_R}|\nabla_{\mathcal {Y}}u_2|^2\bigg|_{t=R^2}dx\leq0.
\end{align}
%$$\iint_{U_R^0}|\partial_t u_2|^2dxdt+\frac{1}{2}\int_{B_R}|\nabla_{\mathcal {Y}}u_2|^2\bigg|_{t=s}dx=\iint_{U_R^0}f\partial_tu_2dxdt.$$
 Taking $\varepsilon$ small enough in  \eqref{8.29}, it follows from \eqref{8.30} that
 \begin{equation}\label{12.6}
  \iint_{U_R^0}|\partial_t u_2|^2dxdt\leq C\iint_{U_R^0}f^2dxdt.
 \end{equation}
Note that, by Lemma \ref{A.3.}, we have
 \begin{align}\label{12.4}
  \iint_{U_R^0}f^2dxdt &\leq \iint_{U_R^0}\big(f-\bar{f}_R\big)^2dxdt+\iint_{U_R^0}(\bar{f}_R)^2dxdt \\\nonumber
    &\leq  CR^{Q+2+2\alpha}[f]^2_{C^{0,\alpha}_{\mathcal {Y}} (U^0_{R})}+CR^{Q+2}\|f\|^2_{L^{\infty}(U^0_{R})}.
 \end{align}
where $\bar{f}_R$ is the average of $f$ on $U^0_R$.

Combined with \eqref{12.6} and \eqref{12.4}, we have
\begin{align}\label{10.01}
  \iint_{U_R^0}|\partial_t u_2|^2dxdt  &\leq CR^{Q+2+2\alpha}[f]^2_{C^{0,\alpha}_{\mathcal {Y}} (U^0_{R})}+CR^{Q+2}\|f\|^2_{L^{\infty}(U^0_{R})}.
\end{align}

For the case $0<\rho<\frac{R}{4}$, consider that $u_2\big|_{t=0}=0$. It follows from Corollary \ref{03.2}, \eqref{8.26} and \eqref{10.01} that
\begin{align}\label{8.3}
\iint_{U^0_\rho}|\nabla^2_{\mathcal {Y}}u_2|^2dxdt &\leq \iint_{U^0_{\frac{R}{4}}}|\nabla^2_{\mathcal {Y}}u_2|^2dxdt\\\nonumber
&\leq \frac{C}{R^4}\iint_{U^0_{R}}u^2_2dxdt+C\iint_{U^0_{R}}f^2dxdt\\\nonumber
 & \leq  C\iint_{U^0_{R}}|\partial_tu_2|^2 dxdt+C\iint_{U^0_{R}}f^2dxdt\\\nonumber
&\leq CR^{Q+2+2\alpha}[f]^2_{C^{0,\alpha}_{\mathcal {Y}} (U^0_{R})}+CR^{Q+2}\|f\|^2_{L^{\infty}(U^0_{R})}.
\end{align}

For the case $\frac{R}{4}<\rho<R$, taking $C=4^{Q+4}$, the following inequality is valid.
\begin{equation}\label{8.2}
\iint_{U^0_\rho}|\nabla^2_{\mathcal {Y}}u_2|^2dxdt\leq C\left(\frac{\rho}{R}\right)^{Q+4}\iint_{U^0_R}|\nabla^2_{\mathcal {Y}}u_2|^2dxdt.
\end{equation}
Combined with \eqref{8.3} and \eqref{8.2}, for any $0<\rho< R$, we have
\begin{align*}
&\quad\iint_{U^0_\rho}|\nabla^2_{\mathcal {Y}}u|^2dxdt\\
   &\leq 2\iint_{U^0_\rho}|\nabla^2_{\mathcal {Y}}u_1|^2dxdt+2\iint_{U^0_\rho}|\nabla^2_{\mathcal {Y}}u_2|^2dxdt\\
   &\leq C\left[\left(\frac{\rho}{R}\right)^{Q+4}\iint_{U^0_R}|\nabla^2_{\mathcal {Y}} u|^2dxdt
  +CR^{Q+2+2\alpha}\left([f]^2_{C^{0,\alpha}_{\mathcal {Y}} (U^0_{R})}+\frac{1}{R^{2\alpha}}\|f\|^2_{L^{\infty}(U^0_{R})}\right)\right].
\end{align*}
Hence by the Iterative Lemma \ref{A.4.}, the result follows.
\end{proof}
\begin{Prop}\label{7.2.10}
Under the same assumptions of Proposition \ref{7.2.9}, then for any $0<\rho\leq\frac{R}{4}$, we have
\begin{align*}
  &\quad \iint_{U^0_{\rho}}\big|D_{\mathcal {Y}}^2 u-\overline{D_{\mathcal{Y}}^2u}_\rho\big|^2dxdt\\
  & \leq C\rho^{Q+2+2\alpha}\left(\frac{1}{R^{4+2\alpha}}\|u\|^2_{L^{\infty}(U^0_{R})}+\frac{1}{R^{2\alpha}}\|f\|^2_{L^{\infty}(U^0_{R})}+[f]^2_{C^{0,\alpha}_{\mathcal {Y}} (U^0_{R})}\right).
\end{align*}
\end{Prop}
\begin{proof}
By Proposition \ref{7.2.9} and Corollary \ref{03.2}, we have
\begin{align}\label{8.25}
 &\quad\iint_{U^0_{\rho}}\big|D_{\mathcal {Y}}^2u\big|^2dxdt \\\nonumber &\leq C\rho^{Q+2+2\alpha}\bigg(\frac{1}{R^{Q+2+2\alpha}}\iint_{U^0_{{R}/{4}}}|\nabla^2_{\mathcal {Y}} u|^2dxdt+\frac{1}{R^{2\alpha}}\|f\|^2_{L^{\infty}(U^0_{R})}+[f]^2_{C^{0,\alpha}_{\mathcal {Y}} (U^0_{R})}\bigg) \\\nonumber
  &\leq C\rho^{Q+2+2\alpha}\bigg(\frac{1}{R^{Q+6+2\alpha}}\iint_{U^0_{R}}u^2dxdt+\frac{1}{R^{2\alpha}}\|f\|^2_{L^{\infty}(U^0_{R})}
  +[f]^2_{C^{0,\alpha}_{\mathcal {Y}} (U^0_{R})}\bigg)\\\nonumber
  &\leq C\rho^{Q+2+2\alpha}\left(\frac{1}{R^{4+2\alpha}}\|u\|^2_{L^{\infty}(U^0_{R})}+\frac{1}{R^{2\alpha}}\|f\|^2_{L^{\infty}(U^0_{R})}+[f]^2_{C^{0,\alpha}_{\mathcal {Y}} (U^0_{R})}\right).
\end{align}
It follows from \eqref{8.25} and the Jensen's inequality that
\begin{align*}
 \left(\overline{D_{\mathcal {Y}}^2u}_\rho\right)^2& \leq \frac{1}{|U^0_\rho|}\iint_{U^0_\rho}\big|D_{\mathcal {Y}}^2u\big|^2dxdt\\
  &\leq C\rho^{2\alpha}\left(\frac{1}{R^{4+2\alpha}}\|u\|^2_{L^{\infty}(U^0_{R})}+\frac{1}{R^{2\alpha}}\|f\|^2_{L^{\infty}(U^0_{R})}+[f]^2_{C^{0,\alpha}_{\mathcal {Y}} (U^0_{R})}\right).
\end{align*}
Then we have
\begin{align*}
   &\quad\iint_{U^0_{\rho}}\big|D_{\mathcal {Y}}^2u-\overline{D_{\mathcal{X}}^2u}_\rho\big|^2dxdt\\
   &\leq 2\iint_{U^0_{\rho}}\big|D_{\mathcal {Y}}^2 u\big|^2dxdt+2\iint_{U^0_{\rho}}\overline{D_{\mathcal{X}}^2u}_\rho dxdt \\
  &\leq C\rho^{Q+2+2\alpha}\left(\frac{1}{R^{4+2\alpha}}\|u\|^2_{L^{\infty}(U^0_{R})}+\frac{1}{R^{2\alpha}}\|f\|^2_{L^{\infty}(U^0_{R})}+[f]^2_{C^{0,\alpha}_{\mathcal {Y}} (U^0_{R})}\right).
\end{align*}
\end{proof}
\begin{proof}[{\bf Proof of Lemma \ref{1.1}}]
The proof is a blend of the localizing technique as Theorem 8.11.1 in \cite{9}, and the method of freezing the coefficients.
For any $x=(x_1,x_2)\in\mathbb{R}^2$, and $R_0\in(0,\varrho)$ chosen as below. Set
$$U_T^0:=\left(0,\frac{R_0^2}{16}\right)\times\mathbb{R}^2\subset U_T,\ \text{and}\ U_T^\prime:=\left[\frac{R_0^2}{16},T\right)\times \mathbb{R}^2\Subset  U_T.$$
Let
 $$ \mathcal {D}_{R}:=\left\{(t,x_1,x_2)\big|t\in\left(0,\frac{R_0^2}{16}\right), x_1\in\left(-R,R\right),x_2\in\mathbb{R}\right\},$$
$$\mathcal{D}_1=\left(0,\frac{R_0^2}{16}\right)\times \left(-\frac{3R_0}{8},\frac{3R_0}{8}\right)\times\mathbb{R},%\quad \mathcal{D}_2=\mathcal {D}_{{3R_0}/{8}}\setminus\mathcal{D}_1,
\ \mathcal{D}_2=U^0_T\setminus\mathcal {D}_{{3R_0}/{8}}.$$
Divide the entire space $U_T$ and $U_T^0$  as follows
\begin{equation*}
  U_T=U_T^0\bigcup U'_T,\ \text{and}\ U_T^0=\mathcal{D}_1\bigcup\mathcal{D}_2.
\end{equation*}

In Step 1 to Step 3, we discuss the Schauder estimate for the Cauchy problem
\begin{equation}\label{7.3}
\left\{
  \begin{array}{ll}
    \partial_t u-\Delta_{\mathcal {X}}u=f,\qquad (t,x)\in U_T,\\
     u(0,x)=0, \qquad\qquad\ x\in \mathbb{R}^2.
  \end{array}
\right.
\end{equation}
Then we extend the estimate to \eqref{1.8} in Step 4.

 %In Step 1, we obtain the estimates on $U'_T$ and $\mathcal{D}_2$ by interior estimates given in Lemma \ref{Sch} and classical results, respectively. In Step 2, we obtain the estimates on $\mathcal{D}_1$ by using Proposition \ref{7.2.10}. In Step 3, we obtain the estimates on $U_T$ for \eqref{7.3}. In Step 4, we obtain the estimates on $U_T$ for the equation  \eqref{1.8}.
%where $U'_T\Subset  U_T$ given in Lemma \ref{Sch}.
%Now we discuss the estimates on $U_T^0$  with the localizing technique.

\noindent{\bf{Step\ 1.}} First we discuss the local prior estimates on $\mathcal{D}_1$ for $\Delta_{\mathcal {Y}}$ by frozen coefficient method.
If  $u\in C^{\infty}(\bar{U}_T)$ is a solution of the equation \eqref{7.2}, take $R=R_0$, then by Lemma \ref{A.3.} and Proposition \ref{7.2.10}, we get
\begin{equation}\label{7.4}
  \big[D_{\mathcal {Y}}^2u\big]_{C^{0,\alpha}_{\mathcal {Y}} (U^0_{{R_0}/{4}})}\leq C\bigg(\frac{1}{R_{0}^{2+\alpha}}\|u\|_{L^{\infty}(U^0_{R_0})}+\frac{1}{R_{0}^\alpha}\|f\|_{L^{\infty}(U^0_{R_0})}+[f]_{C^{0,\alpha}_{\mathcal {Y}} (U^0_{R_0})}\bigg).
\end{equation}
By the modification method, it is sufficient to illustrate \eqref{7.4} also holds for the solution $u \in C^{2,\alpha}_{\mathcal {Y}}(\bar{U}_T)$, and we omit it.

Then suppose that $u$ is the solution of the equation \eqref{7.3} for any $x_0\in\mathcal {Z}$. By the frozen coefficient method, fixed $x_1^0\in(-R_0,R_0)$, we consider
\begin{equation}\label{FC}
  \partial_t u-\left(Y^2_1+\varsigma^2(x^0_1)Y_2^2\right)u=f+\left(\varsigma^2(x_1)-\varsigma^2(x^0_1)\right)Y_2^2u.
\end{equation}
It follows from \eqref{7.4} that
\begin{align*}
  \big[D_{\mathcal {Y}}^2u\big]_{C^{0,\alpha}_{\mathcal {Y}} (U^0_{{R_0}/{4}})}
  &\leq C_4\bigg(R_0^{\alpha}[D_{\mathcal {Y}}^2u\big]_{C^{0,\alpha}_{\mathcal {Y}} (U^0_{R_0})}+\frac{1}{R_0^{2+\alpha}}\|u\|_{L^{\infty}(U^0_{R_0})}\\
  &\quad+\frac{1}{R_0^\alpha}\|f\|_{L^{\infty}(U^0_{R_0})}+[f]_{C^{0,\alpha}_{\mathcal {Y}} (U^0_{R_0})}\bigg).
\end{align*}
Let $R_0\leq C^{-1/\alpha}_4$, then we have
\begin{align*}
  \big[D_{\mathcal {Y}}^2u\big]_{C^{0,\alpha}_{\mathcal {Y}} (U^0_{{R_0}/{4}})}
  &\leq C\bigg(\frac{1}{R_0^{2+\alpha}}\|u\|_{L^{\infty}(U^0_{R_0})}+\frac{1}{R_0^\alpha}\|f\|_{L^{\infty}(U^0_{R_0})}+[f]_{C^{0,\alpha}_{\mathcal {Y}} (U^0_{R_0})}\bigg).
\end{align*}
By choosing
$$C_{R_0}:=\max\left\{C,CR_0^{-(2+\alpha)},CR_0^{-\alpha},2^{3\alpha+1}R_0^{-\alpha}\right\},$$ we have
\begin{align}\label{4}
  \big[D_{\mathcal {Y}}^2u\big]_{C^{0,\alpha}_{\mathcal {Y}} (U^0_{{R_0}/{4}})}
  &\leq C_{R_0}\left(\|f\|_{C^{0,\alpha}_{\mathcal {Y}} (U^0_{R_0})}+\|u\|_{L^{\infty}(U^0_{R_0})}\right).
\end{align}

\noindent{\bf{Step\ 2.}}
Here we discuss the estimates on $U'_T$ and $\mathcal{D}_2$ by applying a covering argument.
Set
\begin{equation}\label{Uk}
  U_{R_0}(t,x):=\big(t-R^2_0,t+R^2_0\big)\times B_{R_0}(x)\subset U_T.
\end{equation}

Let $\big\{(t_m,x_m)\big\}_{m\geq1}$ be a sequence of points in %$(0,T)\times\mathcal {Z}$ defined in \eqref{Z}
$U'_T$ such that
$$U'_T\subset\bigcup_{k\geq1}U_{R_0/4}(t_m,x_m),\ \text{and}\ U_{R_0/4}(t_m,x_m)\bigcap\{t=0\}=\emptyset.$$
Consider $u$ is the solution of \eqref{7.3}, as in \eqref{FC},  by Lemma \ref{Sch},
we first obtain the interior estimates

\begin{equation}\label{3}
   \|u\|_{C_{\mathcal {Y}}^{2,\alpha}(U_{R_0/4}(t_m,x_m))}\leq C\left(\|f\|_{C_{\mathcal {Y}}^{0,\alpha}(U_T)}+\|u\|_{L^{\infty}(U_T)}\right).
\end{equation}

%Let $\big\{(t_k,x_k)\big\}_{k\geq1}$ be a sequence of points in %$(0,T)\times\mathcal {Z}$ defined in \eqref{Z}
%$\mathcal{D}_2$ such that
%$$\mathcal{D}_2\subset\bigcup_{k\geq1}U_{R_0/4}(t_k,x_k)\subset U'_T,\ \text{and}\ U_{R_0/4}(t_k,x_k)\bigcap\{t=0\}=\emptyset.$$
Let $\big\{x_k\big\}_{k\geq1}$ be a sequence of points in $\mathcal{Z}$ such that
\[\mathcal{D}_{{3R_0}/{16}} \subset \bigcup_{k\geq1}U^0_{R_0/4}(x_k).\]
Let $\big\{x_l\big\}_{l\geq1}$ be a sequence of points in $\mathcal{D}_2$ such that
$$\mathcal{D}_2\subset (U_T^0\setminus \mathcal{D}_{{3R_0}/{16}})\subset \bigcup_{l\geq1}U^0_{R_0/4}(x_l)\subset U_T,\ \text{and}\ U^0_{R_0/4}(x_l)\bigcap\mathcal {Z}=\emptyset,$$
then we have
$$U_T^0\subset\bigcup_{k\geq1}U^0_{R_0/4}(x_k)\bigcup_{l\geq1}U^0_{R_0/4}(x_l).$$
%By the interior estimates given in Lemma \ref{Sch}, for any $U_{R_0/4}(t_k,x_k)\subset U'_T$, we have

By the classical Schauder estimates,  we have
\begin{equation}\label{5}
  \|u\|_{C^{2,\alpha}(U^0_{R_0/4}(x_l))}\leq C\left(\|f\|_{C^{0,\alpha}(U_T\setminus\mathcal {D}_{{R_0}/{8}})}+\|u\|_{L^{\infty}(U_T\setminus\mathcal {D}_{{R_0}/{8}})}\right).
\end{equation}

%For any $U_{{R_0}/{4}}(t_k,x_k)$,

\noindent{\bf{Step\ 3.}} We prove the Schauder estimates on $U_T$ for \eqref{7.3}.
It follows from \eqref{3}, \eqref{5} and \eqref{4} that
\begin{align}\label{7.5}
  \big[D_{\mathcal {Y}}^2u\big]_{C^{0,\alpha}_{\mathcal {Y}} (U_{{R_0}/{4}}(t_0,x_0))}
  &\leq C_{R_0}\left(\|f\|_{C^{0,\alpha}_{\mathcal {Y}} (U_T)}+\|u\|_{L^{\infty}(U_T)}\right).
\end{align}
%Since the case of the $t$-coordinate can be treated in a similar manner,
If $u$ is the solution of \eqref{7.3}, for any $(t,x),(s,y)\in U_T$, if $d_P((t,x),(s,y))\geq \frac{R_0}{8}$, then
\begin{equation}\label{8.01}
  \frac{|D_{\mathcal {Y}}^2u(x)-D_{\mathcal {Y}}^2u(y)|}{d_P((t,x),(s,y))^\alpha}\leq\frac{2^{3\alpha+1}}{R_0^\alpha}\|D_{\mathcal {Y}}^2u\|_{L^\infty(U_T)}.
\end{equation}

If $d_P((t,x),(s,y))< \frac{R_0}{8}$, then there exists$(t_0,x_0)\in\{(t_m,x_m),(0,x_k),(0,x_l),\ k,l\geq1\}$ mentioned in Step 1, such that $x,y\in B_{{R_0}/{4}}(x_0)$.
It follows from \eqref{8.05} and \eqref{7.5} that
\begin{align}\label{8.02}
  \frac{|D_{\mathcal {Y}}^2u(t,x)-D_{\mathcal {Y}}^2u(t,y)|}{d_P((t,x),(s,y))^\alpha} &\leq[D^2_{\mathcal {Y}}u]_{C^{0,\alpha}_{\mathcal {Y}}(U_{{R_0}/{4}}(t_0,x_0))} \\ \nonumber
 &\leq C_{R_0}\left(\|u\|_{L^{\infty}(U_T)}+\|f\|_{C^{0,\alpha}_{\mathcal {Y}} (U_T)}\right).
\end{align}
Combined with \eqref{8.01} and \eqref{8.02}, we have
\begin{align*}
  [D_{\mathcal {Y}}^2u\big]_{C^{0,\alpha}_{\mathcal {Y}} (U_{T})}&=\sup_{(t,x),(s,y)\in U_T}\frac{|D_{\mathcal {Y}}^2u(t,x)-D_{\mathcal {Y}}^2u(s,y)|}{d_P((t,x),(s,y))^\alpha} \\\nonumber
 &\leq C_{R_0}\left(\|u\|_{L^{\infty}(U_T)}+\|f\|_{C^{0,\alpha}_{\mathcal {Y}} (U_T)}+\|D_{\mathcal {Y}}^2u\|_{L^\infty(U_T)}\right).
\end{align*}
By the interpolation inequality given in Corollary \ref{P}, then we have
\begin{equation*}
  \big\|u\|_{C^{2,\alpha}_{\mathcal {Y}} (U_T)}\leq C_{R_0}\left(\|u\|_{L^{\infty}(U_{T})}+\|f\|_{C^{0,\alpha}_{\mathcal {Y}} (U_{T})}\right).
\end{equation*}
By the spaces equivalence induced by the vector field $\mathcal {X}$ and $\mathcal {Y}$ given in Lemma 2.2 of \cite{Nagel85}, we have
\begin{equation}
   \big\|u\|_{C^{2,\alpha}_{\mathcal {X}} (U_T)}\leq C_{R_0}\left(\|u\|_{L^{\infty}(U_{T})}+\|f\|_{C^{0,\alpha}_{\mathcal {X}} (U_{T})}\right).
\end{equation}

\noindent{\bf{Step\ 4.}} We extend the Schauder estimates to \eqref{1.8}.
First, if $u$ is the solution of the variable coefficient equation
\begin{equation*}
\left\{
  \begin{array}{ll}
    \partial_tu-\sum^{2}_{i,j=1}a_{ij}(t,x)X_iX_ju=f,\quad (t,x)\in U_{T},\\
     u(0,x)=0, \qquad\qquad\qquad\qquad\qquad\ \ x\in \mathbb{R}^2.
  \end{array}
\right.
\end{equation*}
 By the same argument, using the freezing coefficient method and the interpolation inequality, we obtain
\begin{align*}
  \big\|u\|_{C^{2,\alpha}_{\mathcal {X}} (U_T)} &\leq C_{R_0}\left(\|u\|_{L^{\infty}(U_{T})}+\|f\|_{C^{0,\alpha}_{\mathcal {Y}} (U_{T})}+\|D^2_{\mathcal {X}}u\|_{C^{0,\alpha}_{\mathcal {X}} (U_{T})}\right)\\
   &\leq C\left(\|u\|_{L^{\infty}(U_{T})}+\|f\|_{C^{0,\alpha}_{\mathcal {X}} (U_{T})}\right).
\end{align*}
% The results \eqref{3.35} follows from a perturbation argument using a partition of unity. %we get the solution $u$ satisfies
%$$[u]_{C^{2,\alpha}_{\mathcal {X}}(U_T)}\leq C\left([f]_{C^{0,\alpha}_{\mathcal {X}}(U_T)}+\|\nabla^2_{\mathcal {X}}u\|_{L^{\infty}(U_T)}\right),$$
%where $\nabla_{\mathcal {X}}^2$ is the Hessian matrix.
%By the interpolation inequality which refers to %Lemma B.2 in \cite{09}
%Lemma 6.3.1 in \cite{9}, we get
%$$\|\nabla^2_{\mathcal {X}}u\|_{L^{\infty}(U_T)}\leq C\left(\varepsilon[u]_{C^{2,\alpha}_{\mathcal {X}}(U_T)}+\|u\|_{L^{\infty}(U_T)}\right),$$
%for any $\varepsilon>0$. Taking $\varepsilon$ small enough, we get \eqref{3.35}.

Second, if $u$ is the solution for the equation \eqref{1.8}, we have
\begin{align*}
  \|u\|_{C^{2,\alpha}_{\mathcal {X}}(U_{T})} & \leq C\bigg(\big\|f-\sum^{2}_{i,j=1}b_{i}(t,x)X_iu-c(t,x)u-\mathcal {H}u_0\big\|_{C^{0,\alpha}_{\mathcal {X}}(U_{T})}+\|u\|_{L^{\infty}(U_{T})}\bigg).
  %&\leq C\left(\|f\|_{C^{0,\alpha}_{\mathcal {X}}(U_{R_0}^{0})}+\|\mathcal {H}u_0\|_{C^{0,\alpha}_{\mathcal {X}}(U_{R_0}^{0})}+\|u\|_{C^{0,\alpha}_{\mathcal {X}}(U_{R_0}^{0})}+\|u\|_{C^{1,\alpha}_{\mathcal {X}}(U_{R_0}^{0})}+\|u\|_{L^{\infty}(U_{R_0}^{0})}\right).
\end{align*}
%\begin{align*}
%  \|u\|_{C^{2,\alpha}_{\mathcal {X}}(U_{T})} & \leq C \left(\|f-b_i(t,x)Xu-c(t,x)u\|_{C^{0,\alpha}_{\mathcal {X}}(U_{T})}+\|u\|_{L^{\infty}(U_{T})}\right) \\
%  &\leq C\left(\|f\|_{C^{0,\alpha}_{\mathcal {X}}(U_{T})}+\|u\|_{C^{1,\alpha}_{\mathcal {X}}(U_{T})}\right).
%\end{align*}
Using the interpolation inequality \eqref{9.28} again, we get
\begin{equation}\label{8.04}
  \big\|u\|_{C^{2,\alpha}_{\mathcal {X}} (U_T)}\leq C_{R_0}\left(\|u\|_{L^{\infty}(U_{T})}+\|\mathcal {H}u_0\|_{C^{0,\alpha}_{\mathcal {X}}(U_T)}+\|f\|_{C^{0,\alpha}_{\mathcal {X}} (U_{T})}\right),
\end{equation}
 and the result follows.
\end{proof}

\subsection{The proof of Theorem \ref{1.2}}
\begin{proof}[{\bf{Proof of Theorem \ref{1.2}}}]
The main idea of this proof is based on the Schauder estimates given in Lemma \ref{1.1} and the modification method.

\noindent{\bf{Step\ 1.}} First, we show the existence of the viscosity solution $u$ of \eqref{eq1.3}. Fixed $m\in C([0,T];\mathcal {P}_{1})$, in the equation \eqref{eq1.3}, the function $\left|\nabla_\mathcal {X}u\right|^2-F(x,\bar{m})$ is convex in $\nabla_\mathcal {X}u$. Hence we say that there exists a viscosity solution of \eqref{eq1.3} if $u$ is a value function defined in \eqref{4.21}, which refers to Theorem 4.3.1 in \cite{Hu}. %Though this is not the usual definition  of a viscosity solution, it is equivalent to the usual one in this case, which can refers to \cite{Fleming}.
 By the Hopf transform, setting
$w=e^{-\frac{u}{2}}$. It is clearly that $u$ is a solution of the  quasi-linear equation \eqref{eq1.3} if and only if $w$ is a solution of the linear equation \eqref{2.0}. Hence, there also exists a viscosity solution of \eqref{2.0}.

\noindent{\bf{Step\ 2.}} Consider that $F, f\in C_{\mathcal {X}}^{0,\alpha}(U_{T})$ which are not smooth, we mollify the equation  \eqref{2.0}. We introduce $\eta^\epsilon(x)=\epsilon^{-2}\eta(x/\epsilon)$ for $\epsilon>0$, and $x\in\mathbb{R}^2$, and $\eta$ is a mollification kernel with $\epsilon\rightarrow0$. Then we have
$$\eta^{\epsilon}\ast \partial_t w-\eta^{\epsilon}\ast \Delta_{\mathcal {X}}w+\frac{1}{2}\eta^{\epsilon}\ast (Fw)=0.$$
Denote $w^{\epsilon}:=\eta^{\epsilon}\ast w\in C_0^{\infty}(U_T)$, then $w^{\epsilon}$ satisfies
\begin{equation}\label{3.36}
\left\{
  \begin{array}{ll}
    \partial_t w^{\epsilon}-\Delta_{\mathcal {X}}w^{\epsilon}+\frac{1}{2}Fw^{\epsilon}=f^{\epsilon},\qquad\ (t,x)\in U_{T},  \\
    w^{\epsilon}(0,x)=\eta^{\epsilon}\ast w(0,x),\qquad\qquad\quad x\in\mathbb{R}^2,
  \end{array}
\right.
\end{equation}
where
\begin{equation}\label{8.08}
f^{\epsilon}:=\eta^{\epsilon}\ast\Delta_{\mathcal {X}}w-\Delta_{\mathcal {X}}w^{\epsilon}+\frac{1}{2}\big(Fw^{\epsilon}-\eta^{\epsilon}\ast(Fw)\big),
\end{equation}
satisfying $f^{\epsilon}\in C^{\infty}(U_T)$, and $\|f^{\epsilon}\|_{L^\infty(U_T)}\leq C$.
It follows from \eqref{eq1.60} that
\begin{align*}
  &\quad\eta^{\epsilon}\ast\Delta_{\mathcal {X}}w-\Delta_{\mathcal {X}}w^{\epsilon}\\
  %&=\int_{\mathbb{R}^2}\eta^{\epsilon}(x-y)\Delta_{\mathcal {X}}w(y)dy-\Delta_{\mathcal {X}}\int_{\mathbb{R}^2}\eta^{\epsilon}(y)w(x-y)dy  \\
  &=\int_{\mathbb{R}^2}\eta^{\epsilon}(x-y)\big(\partial^2_{y_1}+h^2(y_1)\partial^2_{y_2}\big)w(y)dy-\Delta_{\mathcal {X}}\int_{\mathbb{R}^2}\eta^{\epsilon}(y)w(x-y)dy\\
  &=\int_{\mathbb{R}^2}\eta^{\epsilon}(y)\big(\partial^2_{x_1}+h^2(x_1-y_1)\partial^2_{x_2}\big)w(x-y)dy-\int_{\mathbb{R}^2}\eta^{\epsilon}(y)(\partial^2_{x_1}+h^2(x_1)\partial^2_{x_2})w(x-y)dy\\
  &=\int_{\mathbb{R}^2}\eta^{\epsilon}(y)\big(h^2(x_1-y_1)-h^2(x_1)\big)\partial^2_{x_2}w(y)dy.
\end{align*}
Since $y$ is close enough to $x$, and $h$ is continuous, for any $\varepsilon>0$, we have
$$|h^2(x_1-y_1)-h^2(x_1)|=|h(x_1-y_1)+h(x_1)||h(x_1-y_1)-h(x_1)|< C\varepsilon.$$
Hence we have
\begin{equation}\label{3.37}
  |\eta^{\epsilon}\ast\Delta_{\mathcal {X}}w-\Delta_{\mathcal {X}}w^{\epsilon}|<\varepsilon.
\end{equation}
Similarly, we get
\begin{equation}\label{3.38}
 |Fw^{\epsilon}-\eta^{\epsilon}\ast(Fw)|<\varepsilon.
\end{equation}
Since $w^{\epsilon}\in C^{2,\alpha}_{\mathcal {X}}(U_T)$, by the prior estimates Lemma \ref{1.1}, we have
\begin{align}\label{3.39}
\|w^{\epsilon}\|_{C^{2,\alpha}_{\mathcal {X}}(\bar{U}_T)}&\leq C\left(\big\|f^{\epsilon}\|_{C^{0,\alpha}_{\mathcal {X}}(\bar{U}_T)}+\|(\Delta_{\mathcal {X}}+F)w^{\epsilon}(0,x)\big\|_{C^{0,\alpha}_{\mathcal {X}}(\bar{U}_T)}+\|w^{\epsilon}\|_{L^{\infty}(U_T)}\right)\\\nonumber
&\leq C\left(\big\|f^{\epsilon}\|_{C^{0,\alpha}_{\mathcal {X}}(\bar{U}_T)}+\|w^{\epsilon}(0,x)\|_{C^{2,\alpha}_{\mathcal {X}}(\bar{U}_T)}+\|w^{\epsilon}\|_{L^{\infty}(U_T)}\right).
\end{align}

Now we claim that
\begin{equation}\label{8.010}
  \|w^{\epsilon}\|_{L^{\infty}(U_T)}\leq C\big(\|f^{\epsilon}\|_{L^{\infty}(U_T)}+\|w^{\epsilon}(0,x)\|_{L^{\infty}(U_T)}\big).
\end{equation}
Without loss of generality, assume that $F\geq2$. Otherwise, we can make transformations $e^{\frac{t}{2}\|F\|_{L^\infty(U_T)}+2t}w^{\epsilon}$ as the proof of Lemma \ref{WMP4}, and we omit it here. To use Lemma \ref{WMP4}, let $\tilde{w}^{\epsilon}:=w^{\epsilon}-\|f^{\epsilon}\|_{L^{\infty}(U_T)}$, then it satisfies
\begin{equation}\label{8.09}
\left\{
  \begin{array}{ll}
    \partial_t \tilde{w}^{\epsilon}-\Delta_{\mathcal {X}}\tilde{w}^{\epsilon}+\frac{1}{2}F\tilde{w}^{\epsilon}=f^{\epsilon}-\frac{1}{2}F\|f^{\epsilon}\|_{L^{\infty}(U_T)}\leq0,\ \ (t,x)\in U_{T},  \\
    \tilde{w}^{\epsilon}(0,x)=\eta^{\epsilon}\ast w(0,x)-\|f^{\epsilon}\|_{L^{\infty}(U_T)},\qquad\qquad\qquad\ \ x\in\mathbb{R}^2.
  \end{array}
\right.
\end{equation}
Applying Lemma \ref{WMP4} to $\tilde{w}^{\epsilon}$, we have
\begin{align}\label{8.11}
  \|\tilde{w}^{\epsilon}\|_{L^{\infty}(U_T)} & \leq C\left(\|f^{\epsilon}\|_{L^{\infty}(U_T)}+\|w^{\epsilon}(0,x)\|_{L^{\infty}(U_T)}\right),
\end{align}
and the claim \eqref{8.010} is valid.

Therefore, we get a uniform estimates from \eqref{8.08}-\eqref{8.010} that
\begin{align*}
  \|w^{\epsilon}\|_{C^{2,\alpha}_{\mathcal {X}}(\bar{U}_T)} & \leq C\left(\|w^{\epsilon}(0,x)\|_{C^{2,\alpha}_{\mathcal {X}}(\bar{U}_T)}+\big\|f^{\epsilon}\|_{C^{0,\alpha}_{\mathcal {X}}(\bar{U}_T)}\right).
\end{align*}
Hence $\{w^{\epsilon}\}$ is a Cauchy sequence in $C^{2,\alpha}_{\mathcal {X}}(\bar{U}_T)$, and converges locally uniformly to $w$ as $\epsilon\rightarrow0$. Thus we conclude that $w\in C^{2,\alpha}_{\mathcal {X}}(\bar{U}_T)$ and holds
\begin{equation}\label{21.6}
   \|w\|_{C^{2,\alpha}_{\mathcal {X}}(\bar{U}_T)} \leq C\|G\|_{C^{2,\alpha}_{\mathcal {X}}(\bar{U}_T)}\leq C,
\end{equation}
because of the assumption {\bf{(H4)}} for $G$.

\noindent{\bf{Step\ 3.}} To prove \eqref{21.2} for $u$, it is necessary to prove $w$ is a positive solution.
%Consider that $w\in C^{2,\alpha}_{\mathcal {X}}(U_{T})$, we deduce from Lemma \ref{1.1} and Lemma \ref{WMP4} that
%\begin{align}\label{21.6}
%  \|w\|_{C^{2,\alpha}_{\mathcal {X}}(U_{T})}&\leq C\left(\|f\|_{C_{\mathcal {X}}^{0,\alpha}(U_{T})}+\|w\|_{L^{\infty}(U_T)}\right)\leq C\left(\|F\|_{C_{\mathcal {X}}^{0,\alpha}(U_{T})}+\|G\|_{C_{\mathcal {X}}^{2,\alpha}(U_{T})}\right).
%\end{align}
%It follows from that we get $[w]_{C^{2,\alpha}_{\mathcal {X}}(U_{T})}\leq C$.
Let $$w^\delta=\delta -w e^{t\frac{\|F\|_{L^{\infty}(U_{T})}}{2}},\ \text{and}\  \delta=e^{-\frac{1}{2}\max_{x\in\mathbb{R}^2}|G|}>0.$$
By calculation, we show that $w^\delta$ satisfies
 $$ \partial_t w^\delta-\Delta_{\mathcal {X}}w^\delta+c(t,x)w^\delta=\delta c(t,x)\leq0,$$
 where $c(t,x)=\frac{1}{2}\left(F-\|F\|_{L^{\infty}(U_{T})}\right)\leq0$.
 Using Lemma \ref{WMP4} for $w^\delta$, we get
 $$\max_{(t,x)\in U_{T}} w^\delta(t,x)\leq\max_{x\in\mathbb{R}^2} (w^\delta(0,x))^{+}=\max_{x\in\mathbb{R}^2}(\delta-e^{-\frac{G}{2}})^+=0.$$
Then $0<\delta\leq w e^{t\frac{\|F\|_{L^{\infty}(U_{T})}}{2}}$.
 Since $F$ is bounded, then $w>0$ for any $t\in(0,T)$.

 Furthermore, consider that $u=-2\ln w$, and
$$  D_{\mathcal {X}}u=-2w^{-1}D_{\mathcal {X}}w,\quad D_{\mathcal {X}}^2u=2w^{-1}D_{\mathcal {X}}^2w+2w^{-2}D_{\mathcal {X}}w,$$
combined with inequalities \eqref{21.6}, we have
\begin{align*}
  \|u\|_{C^{2,\alpha}_{\mathcal {X}}(\bar{U}_T)} &\leq C\|w\|_{C^{2,\alpha}_{\mathcal {X}}(\bar{U}_T)}\leq C.
\end{align*}

\noindent{\bf{Step\ 4.}}
The solution is unique by the compare principle given in Proposition \ref{CP}. Indeed, suppose that $w_1$ and $w_2$ are solutions of the equation \eqref{2.0},  let $\tilde{w}=w_1-w_2$, then $\tilde{w}$ is also a solution of \eqref{2.0} with $\tilde{w}(0,x)=0$. Then we have $\tilde{w}\equiv0$, it implies that $u$ is unique, the result follows.
\end{proof}

\section{EXISTENCE AND UNIQUENESS OF THE FPE}
In this section, we prove the existence and uniqueness of the FPE \eqref{eq1.4} given in Theorem \ref{1.3}. To begin with this process, we obtain some estimates for the following auxiliary equation
 \begin{equation}\label{10.222}
\left\{
 \begin{array}{ll}
  \partial_t m-(\Delta_{\mathcal {X}}+\varepsilon\Delta) m=\hat{f},\qquad\ (t,x)\in U_{T}, \\
  m(0,x)=0,\qquad\qquad\qquad\qquad\ x\in\mathbb{R}^2,
 \end{array}
\right.
\end{equation}
where $\varepsilon>0$.

%If we extend the version of WMP with first order term, then we can get $m\in C^{2,\alpha}_{\mathcal {X}}(U_T)$, and the proof is similar to Theorem \ref{1.2}. However, we shall give the other derivation with lower regularity $m\in C^{2,\frac{\alpha}{\kappa+1}}_{\mathcal {X}}(U_T)$, which is sufficient to ensure Theorem \ref{1.4}.

\begin{Lem}\label{L2}
Let $m^{\varepsilon}$ be the solution of \eqref{10.222}. If $\hat{f}\in L^2(U_T)$, then we have
\begin{equation}\label{10.12}
  \|\varepsilon\Delta m^{\varepsilon}\|_{L^2(U_T)}\leq C\left(\|\hat{f}\|_{L^2(U_T)}+\varepsilon\|\nabla^2_{\mathcal {X}} m^{\varepsilon}\|_{L^2(U_T)}\right).
\end{equation}
Furthermore, for any $\alpha\in(0,1)$, if $\hat{f}\in C^{0,\alpha}_{\mathcal {X}}(U_T)$, then we have
\begin{equation}\label{10.13}
    \|\varepsilon\Delta m^{\varepsilon}\|_{C_{\mathcal {X}}^{0,\alpha}(U_T)}\leq C\left(\|\hat{f}\|_{C_{\mathcal {X}}^{0,\alpha}(U_T)}+\varepsilon\| m^{\varepsilon}\|_{C_{\mathcal {X}}^{2,\alpha}(U_T)}\right).
\end{equation}
\end{Lem}
\begin{proof}
First, we prove the estimates \eqref{10.12}. Square the equation \eqref{10.222}, and integrate on $U_T$, we get
\begin{align}\label{12.1}
\iint_{U_T} \hat{f}^2dxdt&=\iint_{U_T}\big(\partial_tm^{\varepsilon}-(\Delta_{\mathcal {X}}m^{\varepsilon}+\varepsilon\Delta m^{\varepsilon})\big)^2dxdt\\\nonumber
&=\iint_{U_T}(\partial_tm^{\varepsilon})^2-2\partial_tm^{\varepsilon}\big(\Delta_{\mathcal {X}}m^{\varepsilon}+\varepsilon\Delta m^{\varepsilon}\big)+\big(\Delta_{\mathcal {X}}m^{\varepsilon}+\varepsilon\Delta m^{\varepsilon}\big)^2dxdt.
\end{align}
Since $m^{\varepsilon}(0,x)=0$, using integration by parts, we get
\begin{equation}\label{11.1}
-\iint_{U_T}2\partial_tm^{\varepsilon}\Delta_{\mathcal {X}}m^{\varepsilon}dxdt=\iint_{U_T}\partial_t|\nabla_{\mathcal {X}}m^{\varepsilon}|^2dxdt=\int_{\mathbb{R}^2}|\nabla_{\mathcal {X}}m^{\varepsilon}|^2\bigg|_{t=T}dx\geq0.
\end{equation}
Similarly, we get
\begin{equation}\label{11.2}
-\iint_{U_T}2\partial_tm^{\varepsilon}\Delta m^{\varepsilon}dxdt%=\iint_{U_T}\partial_t|\nabla m^{\varepsilon}|^2dxdt
=\int_{\mathbb{R}^2}|\nabla m^{\varepsilon}|^2\bigg|_{t=T}dx\geq0.
\end{equation}
Since
\begin{align*}
 (\Delta_{\mathcal {X}}m^{\varepsilon})^2+2\varepsilon h^2\partial^2_{x_1}m^{\varepsilon}\partial^2_{x_2}m^{\varepsilon} =\big((1+\varepsilon)\partial^2_{x_1}m^{\varepsilon}+h^2\partial^2_{x_2}m^{\varepsilon}\big)^2-2\varepsilon(1+\varepsilon)(\partial^2_{x_1}m^{\varepsilon})^2,
\end{align*}
and use integration by parts, we get
\begin{align}\label{12.2}
 \iint_{U_T}(\Delta_{\mathcal {X}}m^{\varepsilon}+\varepsilon\Delta m^{\varepsilon})^2dxdt
%% &=\iint_{U_T}(\Delta_{\mathcal {X}}m^{\varepsilon})^2+(\varepsilon\Delta m^{\varepsilon})^2+2\varepsilon\Delta_{\mathcal {X}}m^{\varepsilon}\Delta m^{\varepsilon}dxdt\\
 &=\iint_{U_T}(\varepsilon\Delta m^{\varepsilon})^2 +2\varepsilon(h\partial^2_{x_2}m^{\varepsilon})^2+2\varepsilon(\partial_{x_1}\partial_{x_2}m^{\varepsilon})^2\\\nonumber\nonumber
 &\quad+ 2\varepsilon(\partial^2_{x_1}m^{\varepsilon})^2+ 2\varepsilon h^2\partial^2_{x_1}m^{\varepsilon}\partial^2_{x_2}m^{\varepsilon}+(\Delta_{\mathcal {X}}m^{\varepsilon})^2 dxdt\\\nonumber
 &=\iint_{U_T}(\varepsilon\Delta m^{\varepsilon})^2+2\varepsilon(h\partial^2_{x_2}m^{\varepsilon})^2+2\varepsilon(\partial_{x_1}\partial_{x_2}m^{\varepsilon})^2\\\nonumber
 &\quad+\big((1+\varepsilon)\partial^2_{x_1}m^{\varepsilon}+h^2\partial^2_{x_2}m^{\varepsilon}\big)^2-\varepsilon^2(\partial^2_{x_1}m^{\varepsilon})^2dxdt.
\end{align}
It follows from \eqref{12.1}-\eqref{12.2} that
\begin{align}\label{12.3}
  \iint_{U_T}(\varepsilon\Delta m^{\varepsilon})^2dxdt&\leq\iint_{U_T} \hat{f}^2dxdt+\varepsilon^2\iint_{U_T}(\partial^2_{x_1}m^{\varepsilon})^2dxdt\\\nonumber
  & \leq\iint_{U_T} \hat{f}^2dxdt+\varepsilon^2\iint_{U_T}|\nabla^2_{\mathcal {X}}m^{\varepsilon}|^2dxdt,
\end{align}
for any $0<\varepsilon\leq1$, and \eqref{10.12} holds.

Next, we prove the estimate \eqref{10.13}. For any $x\in\mathbb{R}^2$, $R>0$, $U_R$ defined in \eqref{Uk} such that $U_R\subset U_T$. Note that $\hat{f}\in C^{0,\alpha}_{\mathcal {X}}(U_T)$, it implies that $\hat{f}\in L^2(U_R)$. Similar to \eqref{12.4}, it follows from \eqref{12.3} that
\begin{align}\label{12.8}
  \iint_{U_R}(\varepsilon\Delta m^{\varepsilon})^2dxdt&\leq  CR^{Q+2+2\alpha}[\hat{f}]^2_{C^{0,\alpha}_{\mathcal {X}} (U_{R})}+CR^{Q+2}\|\hat{f}\|^2_{L^{\infty}(U_{R})}\\\nonumber
   & \quad+ C\varepsilon^2R^{Q+2+2\alpha}[\nabla^2_{\mathcal {X}}m^{\varepsilon}]^2_{C^{0,\alpha}_{\mathcal {X}} (U_{R})}+C\varepsilon^2R^{Q+2}\|\nabla^2_{\mathcal {X}}m^{\varepsilon}\|^2_{L^{\infty}(U_{R})}.
\end{align}
Since
\begin{align}\label{12.9}
(\overline{{\Delta m^{\varepsilon}}}_R)^2 &=\frac{1}{|U_R|^2}\left(\iint_{U_R}\Delta m^{\varepsilon}dxdt\right)^2 \leq\frac{1}{R^{Q+2}}\iint_{U_R}(\Delta m^{\varepsilon})^2dxdt,
\end{align}
it follows from \eqref{12.8} and \eqref{12.9} that
\begin{align*}
  \iint_{U_R}\big(\varepsilon\Delta m^{\varepsilon} -\varepsilon\overline{{\Delta m^{\varepsilon}}}_R\big)^2dxdt\ &\leq2\iint_{U_R}(\varepsilon\Delta m^{\varepsilon})^2dxdt +2\iint_{U_R}\big(\varepsilon\overline{{\Delta m^{\varepsilon}}}_R\big)^2dxdt  \\
   & \leq  CR^{Q+2+2\alpha}\bigg([\hat{f}]^2_{C^{0,\alpha}_{\mathcal {X}} (U_{R})}+\frac{1}{R^{2\alpha}}\|\hat{f}\|^2_{L^{\infty}(U_{R})}\\
   &\quad+\varepsilon^2[\nabla^2_{\mathcal {X}}m^{\varepsilon}]^2_{C^{0,\alpha}_{\mathcal {X}} (U_{R})}+\frac{\varepsilon^2}{R^{2\alpha}}\|\nabla^2_{\mathcal {X}}m^{\varepsilon}\|^2_{L^{\infty}(U_{R})}\bigg).
\end{align*}
Thus by Lemma \ref{A.3.}, we have
\begin{align*}
  &\quad[\Delta m^{\varepsilon}]^2_{C^{0,\alpha}_{\mathcal {X}} (U_{R})}\\
   &\leq C\left( [\hat{f}]^2_{C^{0,\alpha}_{\mathcal {X}} (U_{R})}+\frac{1}{R^{2\alpha}}\|\hat{f}\|^2_{L^{\infty}(U_{R})}+\varepsilon^2[\nabla^2_{\mathcal {X}}m^{\varepsilon}]^2_{C^{0,\alpha}_{\mathcal {X}} (U_{R})}+\frac{\varepsilon^2}{R^{2\alpha}}\|\nabla^2_{\mathcal {X}}m^{\varepsilon}\|^2_{L^{\infty}(U_{R})}\right).
\end{align*}
Fixed $R>0$, by the interpolation inequality given in Corollary \ref{P}, then we have
 \begin{align*}
   \|\varepsilon\Delta m^{\varepsilon}\|_{C_{\mathcal {X}}^{0,\alpha}(U_R)} & \leq C\left(\|\hat{f}\|_{C_{\mathcal {X}}^{0,\alpha}(U_R)}+\varepsilon\|\nabla^2_{\mathcal {X}} m^{\varepsilon}\|_{C_{\mathcal {X}}^{0,\alpha}(U_R)}\right) \\
   & \leq C\left(\|\hat{f}\|_{C_{\mathcal {X}}^{0,\alpha}(U_T)}+\varepsilon\|m^{\varepsilon}\|_{C_{\mathcal {X}}^{2,\alpha}(U_T)}\right),
 \end{align*}
 for all $U_R$,  and \eqref{10.13} follows.
\end{proof}

Let $u$ be the solution of HJE \eqref{eq1.3}, for the operator $\mathcal {H}$ defined in \eqref{001}, with
\begin{equation}\label{8}
A=\begin{pmatrix}1&0\\0&1\\ \end{pmatrix},\ b(t,x)=\nabla_{\mathcal {X}}u(t,x),\  c(t,x)=\Delta_{\mathcal {X}}u(t,x).
\end{equation}
By Theorem \ref{1.2}, we have
\begin{equation}\label{08}
b(t,x)\in C^{1,\alpha}_{\mathcal {X}}(U_{T}),\ \text{and}\ c(t,x)\in C^{0,\alpha}_{\mathcal {X}}(U_{T}).
\end{equation}
Given the optimal control $\alpha^*$ defined in \eqref{11.20}, then FPE \eqref{eq1.4} can be regarded as a linear degenerate equation under the operator $\mathcal {H}$. Now we prove Theorem \ref{1.3}.
\begin{proof}
[\bf{Proof of Theorem \ref{1.3}}]
In this proof, by Cantor's diagonal argument, we apply Arzela-Ascoli Theorem to obtain the regularity of $m$.%, which bases on the uniform estimates for $m^\varepsilon$.

\noindent{\bf{Step\ 1.}} We show the existence of the distribution solution by vanishing viscosity method. For all $\varepsilon>0$, we consider the equation
 \begin{equation}\label{10.22}
\left\{
 \begin{array}{ll}
  \partial_t m-(\Delta_{\mathcal {X}}+\varepsilon\Delta) m-b(t,x)\nabla_{\mathcal {X}}m- c(t,x)m=\tilde{f},\quad (t,x)\in U_{T}, \\
  m(0,x)=0,\qquad\qquad\qquad\qquad\qquad\qquad\qquad\qquad\quad\ \ x\in\mathbb{R}^2,
 \end{array}
\right.
\end{equation}
%where $b_n(t,x),\ c_n(t,x)\in C^{\infty}(U_{T})$ are the mollification of
where $\tilde{f}=(\Delta_{\mathcal {X}}+\varepsilon\Delta) m_0+b(t,x)\nabla_{\mathcal {X}}m_0+c(t,x)m_0$, $b(t,x)$ and $c(t,x)$  are given in \eqref{8}.
Since the assumption {\bf{(H6)}} and \eqref{08}, it is clearly that $\tilde{f}\in L^2(U_{T})$, and the assumption {\bf{(H1)}} implies that the smooth function $h\in C^{0,\alpha}(\bar{U}_{T})$. By the classical theory of parabolic equations, there exists a unique classical solution $m^\varepsilon$ of \eqref{10.22}, and $m^\varepsilon\in C^{2,\frac{\alpha}{\kappa+1}}(\bar{U}_{T})$. For any test function $\varphi\in C^\infty_0(U_T)$, we have
$$\int^T_0\int_{\mathbb{R}^2}m^\varepsilon\left(-\partial_t\varphi-(\Delta_{\mathcal {X}}+\varepsilon\Delta)\varphi+{\rm{div}}_{\mathcal {X}}\big(b(t,x)\varphi\big)-c(t,x)\varphi\right) dxd\tau=0.$$
Moreover,
\begin{equation}\label{eq13.5}
  \|m^{\varepsilon}\|_{L^\infty(U_{T})}\leq C\|\tilde{f}\|_{L^\infty(U_{T})}.
\end{equation}
This implies that up to a subsequence still denoted by $m^\varepsilon$, such that $m^\varepsilon$ converges weakly to $m$ in weak*-topology $L_{loc}^\infty(U_{T})$, when $\varepsilon\rightarrow0$. Therefore $m$ is a distribution solution of equation \eqref{eq1.4}.

\noindent{\bf{Step\ 2.}} We obtain the uniform estimates for $m^\varepsilon$.
Since $m^\varepsilon\in C^{2,\frac{\alpha}{\kappa+1}}(U_{T})$, by the definition \eqref{8.05} and \eqref{2.7}, we have $m^\varepsilon\in C^{2,\frac{\alpha}{\kappa+1}}_{\mathcal {X}}(U_{T})$.
 By Lemma \ref{1.1}, we get
\begin{equation}\label{3.05}
 \|m^\varepsilon\|_{C^{2,\frac{\alpha}{\kappa+1}}_{\mathcal {X}}(U_T)}\leq C\left( \|\tilde{f}+\varepsilon\Delta m^\varepsilon\|_{C_{\mathcal {X}}^{0,\frac{\alpha}{\kappa+1}}(U_T)}+\|m^\varepsilon\|_{L^{\infty}(U_T)}\right).
\end{equation}
It follows from Lemma \ref{L2} that
 \begin{align}\label{10.015}
   &\quad\|\varepsilon\Delta m^\varepsilon\|_{C_{\mathcal {X}}^{0,\frac{\alpha}{\kappa+1}}(U_T)}\\\nonumber
   % &\leq \|\varepsilon\Delta m^\varepsilon\|_{C_{\mathcal {X}}^{0,\frac{\alpha}{\kappa+1}}(\bar{U}_T)}\\\nonumber
     &\leq C\|\tilde{f}+b(t,x)\nabla_{\mathcal {X}}m^\varepsilon+ c(t,x)m^\varepsilon\|_{C_{\mathcal {X}}^{0,\frac{\alpha}{\kappa+1}}(U_T)}+\varepsilon \|m^\varepsilon\|_{C_{\mathcal {X}}^{2,\frac{\alpha}{\kappa+1}}(U_T)}.
     %&\leq C\|\tilde{f}\|_{C_{\mathcal {X}}^{0,\frac{\alpha}{\kappa+1}}(U_T)}+C\varepsilon \| m^\varepsilon\|_{C_{\mathcal {X}}^{2,\frac{\alpha}{\kappa+1}}(\bar{U}_T)}.
 \end{align}
It follows from the classical WMP for $m^\varepsilon$ that
\begin{equation}\label{10.014}
\|m^\varepsilon\|_{L^{\infty}(U_{T})}\leq C\|\tilde{f}\|_{L^{\infty}(U_{T})}.
\end{equation}
It follows from \eqref{3.05} and \eqref{10.015} that
\begin{equation}\label{12.31}
 \|m^\varepsilon\|_{C^{2,\frac{\alpha}{\kappa+1}}_{\mathcal {X}}(U_T)}\leq C\|\tilde{f}+b(t,x)\nabla_{\mathcal {X}}m^\varepsilon+ c(t,x)m^\varepsilon\|_{C_{\mathcal {X}}^{0,\frac{\alpha}{\kappa+1}}(U_T)}+C\varepsilon \| m^\varepsilon\|_{C_{\mathcal {X}}^{2,\frac{\alpha}{\kappa+1}}(U_T)}.
\end{equation}

By the interpolation inequality, and then taking $\varepsilon$ small enough, we have
\begin{align}\label{12.7}
   \|m^\varepsilon\|_{C^{2,\frac{\alpha}{\kappa+1}}_{\mathcal {X}}(U_T)}&\leq C\big(\|\tilde{f}\|_{C^{0,\frac{\alpha}{\kappa+1}}(U_T)}+\varepsilon \| m^\varepsilon\|_{C_{\mathcal {X}}^{2,\frac{\alpha}{\kappa+1}}(U_T)}+\|m^\varepsilon\|_{L^{\infty}(U_{T})}\big)\\\nonumber
   &\leq C\|\tilde{f}\|_{C^{0,\frac{\alpha}{\kappa+1}}(U_T)},
\end{align}
where $C$ is a positive constant independent on $\varepsilon$.

\noindent{\bf{Step\ 3.}} We show that $m\in C^{2,\frac{\alpha}{\kappa+1}}_{\mathcal {X}}(\bar{U}_T)$ by Cantor's diagonal argument.
Consider that there exists a sequence $\{x_k\}\in\mathbb{R}^2$, such that $$\mathbb{R}^2=\bigcup_{k\geq1}B_R(x_k).$$
 Fixed $R>0$, for any  $(t_k,x_k)\in U_T$, $U^k_R:=(t_k-R^2,t_k+R^2)\times B_R(x_k)$ such that $U^k_R\subset U_T$.
  It follows from \eqref{12.31} that
  $$   \|m^\varepsilon\|_{C^{2,\frac{\alpha}{\kappa+1}}_{\mathcal {X}}(\bar{U}^k_R)}\leq \|m^\varepsilon\|_{C^{2,\frac{\alpha}{\kappa+1}}_{\mathcal {X}}(\bar{U}_T)}\leq C\|\tilde{f}\|_{C^{0,\frac{\alpha}{\kappa+1}}(\bar{U}_T)}.$$
 Using the Arzela-Ascoli Theorem on $\bar{U}^1_{R}$, we first get a subsequence $\{m^{\varepsilon_{n_1}}\}$ converges uniformly to $m$ as $n_1\rightarrow\infty$. Moreover, $ m\in C^{2,\frac{\alpha}{\kappa+1}}_{\mathcal {X}}(\bar{U}^1_{R})$. Next, we take a subsequence of $\{m^{\varepsilon_{n_1}}\}$ denoted as $\{m^{\varepsilon_{n_2}}\}$, such that it converges uniformly to $m$ on $\bigcup^{2}_{k=1}\bar{U}^k_{R}$ as $n_2\rightarrow\infty$. And so on, use a standard diagonal argument to select a sequence $\{m^{\varepsilon_{n_k}}\}$,
 which converges uniformly to $m$ on $U_T$ as $n_k\rightarrow\infty$. We thus conclude that $m\in C^{2,\frac{\alpha}{\kappa+1}}_{\mathcal {X}}(\bar{U}_T)$, the result follows.
\end{proof}

\section{EXISTENCE AND UNIQUENESS OF THE MFG}
In this section, we prove the existence and uniqueness of the MFG systems \eqref{eq1.1}. The uniqueness holds under classical hypothesis on the monotonicity of $F$ and $G$ given in assumption {\bf(H5)}, which can refer to Theorem 4.3 in \cite{Cardaliaguet}. Now we only prove the existence of the MFG systems \eqref{eq1.1}, and the proof is based on the Schauder fixed point theorem given in Lemma \ref{APP1.33} .
\begin{proof}[\bf{Proof of Theorem \ref{1.4}}]
For any probability measures $\mu\in\mathcal {C}$, we associate $m=\psi(\mu)\in\mathcal {C}$ in the following way. Let $u$ is a unique solution to the terminal problem
\begin{equation}\label{3.3}
\left\{
 \begin{array}{ll}
 -\partial_t u-\Delta_{\mathcal {X}} u+\frac{1}{2}| \nabla_{\mathcal {X}} u|^2=F(x,\mu),\quad\ (t,x)\in U_{T},\\
  u(T,x)=G(x,\mu_T),\qquad\qquad\qquad\qquad\ x\in\mathbb{R}^2.
 \end{array}
\right.
\end{equation}
Then we define $m:=\psi(\mu)$ as the solution of the FPE
\begin{equation}\label{3.4}
\left\{
 \begin{array}{ll}
    \partial_t m-\Delta_\mathcal {X} m-\nabla_\mathcal {X}u \nabla_\mathcal {X}m-\Delta_\mathcal {X}um=0,\quad\ (t,x)\in U_{T},\\
    m(0,x)=m_0,\ \ \qquad\qquad\qquad\qquad\qquad\qquad\ x\in\mathbb{R}^2.
 \end{array}
\right.
\end{equation}

 First, let us check that $\psi$ is a well-defined from $\mathcal {C}$ to itself. %Because of Theorem \ref{1.2} and Theorem \ref{1.3}, the mapping $\psi$ is single value. And
 Under assumptions {\bf{(H1)}}-{\bf(H4)},
 Theorem \ref{1.2} shows that the HJE \eqref{3.3} has a unique solution $u$ belongs to $C^{2,\alpha}_{\mathcal {X}}(U_{T})$. Moreover, we have an estimate \eqref{21.2}.
 %Recall that the maps $x\to F(x; \mu(t))$ and $x\to G(x; \mu(T))$ are bounded by constant $C$, so that a straightforward application of the comparison principle, see Proposition \ref{CP}, which implies that $u$ is bounded. In the same way, since the maps $x\to F(x; \mu(t))$ and $x\to G(x; \mu(T))$ are Lipschitz continuous, then $u$ is also Lipschitz continuous. Combined with the boundness of the function $h$, we show that $\nabla_{\mathcal {X}}u$ is bounded.
%Recall that the maps $x\rightarrow F(x;\mu(t))$ and $x\rightarrow F(x;\mu(T))$ are bounded by $C_0$, so that a straightforward application of the comparison principle Proposition \ref{CP} implies that $u$ is bounded by $(1 + T)C_0$. In the same way, since moreover the maps $x\rightarrow F(x;\mu(t))$ and $x\rightarrow F(x;\mu(T))$ are Lipschitz continuous, $u$ is also Lipschitz continuous. Hence $Xu$ is bounded by C, since $h$ is bounded.
Now, we turn to the FPE \eqref{3.4} with {\bf(H6)}. Since $u\in C^{2,\alpha}_{\mathcal {X}}(U_{T})$, the maps $(t,x)\to \nabla_{\mathcal {X}}u(t,x)$ and $(t,x)\to \Delta_{\mathcal {X}}u(t,x)$ belong to $C^{0,\alpha}_{\mathcal {X}}(U_{T})$. Hence by Theorem \ref{1.3}, there is a unique solution $m$ of the equation \eqref{3.4}, and $m$ belongs to $C^{2,\frac{\alpha}{\kappa+1}}_{\mathcal {X}}(U_T)$. Moreover, from Lemma \ref{APP1.4}-\rm{(i)}, for any $s, t\in[0,T]$, we have the following estimates
$$d_1(m(t,\cdot),m(s,\cdot)) \leq C\|h\|_{L^{\infty}(\mathbb{R})}\big(\|\nabla_{\mathcal {X}}u\|_{L^{\infty}(U_T)}+1\big)|t-s|^{\frac{1}{2}},$$
where $h$ and $\|\nabla_\mathcal {X} u\|_{L^\infty(U_{T})}$ are bounded, because of the uniform estimate \eqref{21.2}.
 Thus by the definition \eqref{eq2.7}, then $m$ belongs to $\mathcal {C}$ , and the mapping $\psi:\mu\rightarrow m:=\psi(\mu)$ is well-defined from $\mathcal {C}$ into itself.%, as long as $C$ defined in \eqref{eq2.7} is big enough.

Second, let us check that $\psi$ is a continuous map. Let $\mu_n\in \mathcal {C}$ converge to some $\mu$. Let $(u_n,m_n)$ and $(u,m)$ be the corresponding solutions to \eqref{3.3}-\eqref{3.4}. By the continuity assumption {\bf{(H3)}} on $F$ and $G$, we get $(t,x)\rightarrow F(x,\mu_n(t))$, $x\rightarrow G(x,\mu_n(T))$ locally uniformly converge to $(t,x)\rightarrow F(x,\mu(t))$, $x\rightarrow G(x,\mu(T))$.
Then one gets the local uniformly converge to $u$ by standard arguments. Since $\{\nabla_{\mathcal {X}}u_n\}_{n}$ are uniformly bounded and the interior regularity result given in Lemma \ref{APP1.4}-\rm{(ii)}, we know that $\{\nabla_{\mathcal {X}}u_n\}_n$ are locally uniformly H\"{o}lder continuous and therefore locally uniformly converges to $\nabla_{\mathcal {X}}u$. This easily implies that any converging subsequence of the relatively compact sequence $m_n$ is a weak solution of \eqref{3.4}. But $m$ is the unique weak solution of the equation \eqref{3.4}, which proves that $\{m_n\}$ converges to $m$.

Because $\mathcal {C}$ is compact, the continuous map $\psi$ is compact. We conclude by Schauder fixed point theorem that the compact map $\mu\rightarrow m=\psi(\mu)$ has a fixed point in $\mathcal {C}$. This fixed point $m$ and its corresponding $u$ is a coupling solution of the MFG systems \eqref{eq1.1}. Then the result follows.
\end{proof}

%\bibliography{ref}
%\nocite{*}

\addcontentsline{toc}{section}{References} %\newpage
%\bibliography{master}
%\bibliographystyle{wileybib}
%\nocite{*}
\end{document}